\documentclass[10pt,DIV=10,a4paper]{scrartcl}

\usepackage[utf8]{inputenc}

\usepackage{hyperref}

\usepackage{amsmath}
\usepackage{amssymb}
\usepackage{mathtools}
\usepackage{amsthm}

\newtheorem{theorem}{Theorem}
\newtheorem{proposition}[theorem]{Proposition}
\newtheorem{lemma}[theorem]{Lemma}
\newtheorem{corollary}[theorem]{Corollary}
\newtheorem{definition}[theorem]{Definition}

\theoremstyle{remark}
\newtheorem{remark}[theorem]{Remark}

\usepackage[backend=bibtex,isbn=false,maxcitenames=5]{biblatex}
\bibliography{../lit.bib}

\newcommand{\conj}[1]{\overline{#1}} 

\newcommand{\Lap}{\mathcal{L}} 
\DeclareMathOperator{\range}{ran} 
\DeclareMathOperator{\adjoint}{adj} 

\newcommand{\tnorm}[1]{{\left\vert\kern-0.25ex\left\vert\kern-0.25ex\left\vert #1
    \right\vert\kern-0.25ex\right\vert\kern-0.25ex\right\vert}}

\newcommand{\conv}{\star}

\newcommand{\sol}[3]{\mathcal{S}^{#1}_{#2\to #3}}

\newcommand{\T}{\mathbb{T}}
\newcommand{\R}{\mathbb{R}}
\newcommand{\C}{\mathbb{C}}
\newcommand{\Z}{\mathbb{Z}}
\newcommand{\N}{\mathbb{N}}
\newcommand{\ee}{\mathrm{e}} 
\newcommand{\ii}{\mathrm{i}} 
\newcommand{\dd}{\mathrm{d}} 
\newcommand{\DD}{\mathrm{D}} 
\newcommand{\stat}{\mathrm{st}} 
\newcommand{\pert}{\mathrm{pt}} 
\newcommand{\init}{\mathrm{in}} 

\newcommand{\compF}{\mathsf{F}}

\newcommand{\Id}{\mathrm{Id}}

\newcommand{\vkV}{\mathcal{V}}
\newcommand{\sX}{\mathcal{X}} 
\newcommand{\sY}{\mathcal{Y}} 

\DeclareMathOperator*{\esssup}{ess\, sup}

\begin{document}
\title{Stability of partially locked states in the Kuramoto model
  through Landau damping with Sobolev regularity}
\author{Helge Dietert\thanks{CNRS, Sorbonne Université, Université Paris Diderot,
  Institut de Mathématiques de Jussieu-Paris Rive Gauche, IMJ-PRG, F-75013, Paris, France\newline
  Email: \texttt{helge.dietert@imj-prg.fr}\newline
  The author thankfully acknowledges support by the ANR Chaire
d'Excellence ANR-11-IDEX-005 and the People Programme (Marie Curie
Actions) of the European Union’s Seventh Framework Programme
(FP7/2007-2013) under REA grant agreement n.  PCOFUND-GA-2013-609102,
through the PRESTIGE programme coordinated by Campus France.}}

\maketitle

\begin{abstract}
  The Kuramoto model is a mean-field model for the synchronisation
  behaviour of oscillators, which exhibits Landau damping. In a recent
  work, the nonlinear stability of a class of spatially inhomogeneous
  stationary states was shown under the assumption of analytic
  regularity. This paper proves the nonlinear Landau damping under the
  assumption of Sobolev regularity. The weaker regularity required the
  construction of a different more robust bootstrap argument, which
  focuses on the nonlinear Volterra equation of the order parameter.
\end{abstract}

\section{Introduction}

The Kuramoto model
\cite{kuramoto-1984-chemical-oscillation,kuramoto-1975-self-entrainment-coupled-oscillators}
is a mean-field model for the interaction of oscillators, which shows
the Landau damping behaviour
\cite{strogatz-mirollo-matthews-1992-coupled-nonlinear-oscillators,
  chiba-2015-proof-kuramoto-conjecture,
  benedetto-caglioti-montemagno-2016-exponential-dephasing,
  fernandez-gerard-varet-giacomin-2016-landau-damping,
  dietert-2016-stability-bifurcation,
  dietert-fernandez-gerard-varet-2018-landau-damping-pls}. On
the particle level the model consists of oscillators $i=1,\dots,N$,
which are modelled by their position $\theta_i$, the phase angle on the
torus $\T = \R / (2\pi\Z)$, and their velocity $\omega_i \in \R$, the
intrinsic frequency. The evolution is determined by the system of ODEs
\begin{equation*}
  \left\{
    \begin{aligned}
      &\frac{\dd}{\dd t} \theta_i = \omega_i + \frac{K}{N}
      \sum_{j=1}^{N} \sin(\theta_j-\theta_i),\\
      &\frac{\dd}{\dd t} \omega_i = 0
    \end{aligned}
  \right.
\end{equation*}
for $i=1,\dots,N$, where $K$ is the coupling constant. The intuition
is that each oscillator $i$ evolves according to its own intrinsic
frequency $\omega_i$ and according to a global coupling, which tries
to synchronise the oscillators.

The overall synchronisation is described by the order parameter
\(r \in \C\) given by
\begin{equation*}
  r = \frac{1}{N} \sum_{j=1}^{N} \ee^{\ii \theta_j}.
\end{equation*}
The order parameter also describes the coupling so that the evolution
can be written in the mean-field form
\begin{equation*}
  \frac{\dd}{\dd t} \theta_i = \omega_i + \frac{K}{2\ii}
  \left( r\, \ee^{-\ii \theta_i} - \conj{r}\, \ee^{\ii \theta_i} \right).
\end{equation*}

For the study of a large number of oscillators ($N\to\infty$), the
mean-field limit is used. In the mean-field limit the system is
described by a measure $f(t,\theta,\omega)\,\dd\theta\dd\omega$ for
the distribution of oscillators and the evolution is described by the
PDE
\begin{equation}
  \label{eq:kuramoto-mean-field-pde}
  \left\{
    \begin{aligned}
      &\partial_t f(t,\theta,\omega)
      + \partial_\theta
      \left(
        \left[
          \omega +
          \frac{K}{2\ii} \left(r(t)\, \ee^{-\ii \theta} -
            \conj{r(t)}\, \ee^{\ii \theta} \right)
        \right]
        f(t,\theta,\omega)
      \right)
      = 0, \\
      &r(t) = \int_{\R} \int_{\T} \ee^{\ii \theta} f(t,\theta,\omega)\,
      \dd \theta\, \dd \omega.
    \end{aligned}
  \right.
\end{equation}
Due to the Lipschitz regular interaction, this limit can be justified
in the mean-field limit framework by
\textcite{braun-hepp-1977-vlasov-dynamics-and-its-fluctuations,
  dobrushin-1979-vlasov,
  neunzert-1984-introduction-nonlinear-boltzmann-vlasov}, see
\cite{sakaguchi-1988-cooperative-phenomena-in-coupled-oscillator,
  strogatz-1991-stability-incoherence-coupled-oscillators,
  lancellotti-2005-vlasov-limit-nonlinear-coupled-oscillators}. In
particular, this shows the well-posedness of the PDE
\eqref{eq:kuramoto-mean-field-pde}.

The evolution preserves the velocity distribution
\begin{equation*}
  g(\omega) = \int_{\T} f(\theta,\omega)\, \dd \theta.
\end{equation*}
Together with the coupling constant \(K\) the velocity distribution
characterises the typical behaviour observed by physicists
\cite{strogatz-2000-kuramoto-to-crawford,dietert-fernandez-2018-asymptotic-stability}:
If \(K\) is sufficiently small, the solution converges to a spatially
homogeneous state, i.e.\ a state not depending on \(\theta\). For
larger \(K\), the system undergoes a bifurcation and converges to a
stationary state \(f_\stat\) with order parameter \(r_\stat \not = 0\)
(after possibly taking out a global drift which amounts to redefining
\(g\)). The behaviour around the spatially homogeneous state is easier
to understand and has been treated in
\cite{chiba-medvedev-2016-preprint-kuramoto-on-graphs,%
  chiba-2015-proof-kuramoto-conjecture,%
  fernandez-gerard-varet-giacomin-2016-landau-damping,%
  dietert-2016-stability-bifurcation,%
  benedetto-caglioti-montemagno-2016-exponential-dephasing}. This work
focuses on stationary states with \(r_\stat \not = 0\).

The system is invariant under the rotation symmetry $R_{\Theta}$
changing $\theta \to \theta + \Theta$, i.e.
\begin{equation*}
  (R_{\Theta} f)(\theta,\omega) = f(\theta+\Theta, \omega).
\end{equation*}
Hence for finding stationary states we can take w.l.o.g.\ \(r_\stat \in [0,1]\).

Given \(r_\stat \in [0,1]\), the distribution \(f_\stat\) of the
stationary state must be
\begin{equation}
  \label{eq:stationary-state-physical}
  f_{\stat}(\theta,\omega) =
  \begin{dcases}
    \delta_{\arcsin(\omega/(Kr_{\stat}))}(\theta)\, g(\omega)
    &\text{if $|\omega|\le K r_{\stat}$}, \\
    \frac{\sqrt{\omega^2 - (Kr_{\stat})^2}}{2\pi |\omega -
      Kr_{\stat}\sin \theta|}\, g(\omega)
    &\text{if $|\omega|>K r_{\stat}$},
  \end{dcases}
\end{equation}
where oscillators with \(|\omega| \le K r_{\stat}\) are locked at the
stable fixed point $\arcsin(\omega/(Kr_{\stat}))$ with
$|\arcsin(\omega/(Kr_{\stat}))| < \pi/2$ (the other fixed-point
$\pi - \arcsin(\omega/(Kr_{\stat}))$ is unstable) and oscillators with
\(|\omega| > K r_{\stat}\) are drifting. This dichotomy gives the name
partially locked states (PLS) to these stationary states.

A choice \(r_{\stat} \in [0,1]\) corresponds to a stationary state if
and only if the self-consistency equation
\begin{equation*}
  r_\stat = \int_{\R} \int_{\T} \ee^{\ii \theta} f_{\stat}(\theta,\omega)\,
  \dd \theta\, \dd \omega
\end{equation*}
is satisfied
\cite{mirollo-strogatz-2007-spectrum-of-pls-for-kuramoto-model,%
  mirollo-strogatz-2007-spectrum-of-pls-for-kuramoto-model,%
  strogatz-2000-kuramoto-to-crawford,%
  omelchenko-wolfrum-2013-bifurcations-sakaguchi-kuramoto}. By
continuity arguments the existence of such states can be assured
\cite{dietert-fernandez-gerard-varet-2018-landau-damping-pls,%
  dietert-fernandez-2018-asymptotic-stability}.

After finding the stationary states, the next question is their
stability under perturbations. Factoring out the rotation symmetry,
the author identified with Fernandez and Gérard-Varet in
\cite{dietert-fernandez-gerard-varet-2018-landau-damping-pls} the
linear stability criterion and showed nonlinear stability, under the
assumption of analytic regularity in $\omega$.

This work extends the analysis to the case of Sobolev regularity of
the stationary state and the perturbation. A main difficulty lies in
the handling of the nonlinearity: in the previous work
\cite{dietert-fernandez-gerard-varet-2018-landau-damping-pls} a linear
theory was developed in a suitable functional analytic setting and the
nonlinearity is treated as perturbation. In the case of Sobolev
regularity, this is no longer possible and we need to device a
bootstrap argument taking into account the structure of the
nonlinearity.

The main idea is to study the evolution of a perturbation under the
transport including the nonlinearity. For the order parameter we find
in the linear case a Volterra convolution equation in time. Instead of
putting the nonlinear effect solely into the forcing, we move a part
of the nonlinearity into the kernel of the Volterra equation by
evolving the perturbations under the nonlinear transport. On the level
of the Volterra equation we do not see the regularity problem and can
close an estimate for the order parameter. This then allows a
bootstrap argument controlling the evolution.

A large part of the technical complications comes from handling the
rotation symmetry. In order to factor it out, we construct a rotation
angle \(\Theta : \R^+ \mapsto \R\) and aim to show that
\(R_{-\Theta(t)} f(t)\) converges to \(f_{\stat}\).

In order to measure the perturbation take the Fourier transform
$\hat{f}$ of $f$, which takes $\theta$ to $\ell$ and $\omega$ to $\xi$
with the convention
\begin{equation*}
  (\hat{f})_\ell(\xi) = \int_{\R} \int_{\T}
  \ee^{-\ii \ell \theta - \ii \xi \omega} f(\theta,\omega)\,
  \dd \theta\, \dd \omega
\end{equation*}
and accordingly
\begin{equation*}
  \hat{g}(\xi) = \int_{\R}
  \ee^{- \ii \xi \omega} g(\omega)\,
  \dd \omega.
\end{equation*}
For \(\hat{f}\) introduce the norm \(\|\cdot\|_{p_b}\) by
\begin{equation*}
  \| u \|^2_{p_b} = \sum_{\ell \ge 1} \int_{\xi=0}^{\infty}
  \Big(
    |u_\ell(\xi)|^2 + |\partial_\xi u_\ell(\xi)|^2
  \Big)
  |p_b(\xi)|^2 \frac{1}{\ell} \dd \xi
\end{equation*}
and accordingly
\begin{equation}
  \label{eq:def-norm-ghat}
  \| \hat{g} \|_{p_b}^2
  = \int_0^\infty
  \Big(
    |\hat{g}(\xi)|^2 + |\partial_\xi \hat{g}(\xi)|^2
  \Big)
  |p_b(\xi)|^2\, \dd \xi,
\end{equation}
where \(p_b(\xi) = (1+\xi)^b\) and \(b \ge 0\). In physical space the
weight corresponds to Sobolev regularity \(H^b\) except that we only
measure non-negative frequencies \(\xi \ge 0\). The advantage of this
measure for the initial perturbation is that
\(\|\hat{f}_{\stat}\|_{p_b} < \infty\) for
\(\| \hat{g} \|_{p_{b+1}} < \infty\), even though, \(f_{\stat}\) is
singular and consists of Dirac measures.

Postponing the linear stability condition, we now formulate the main
result.
\begin{theorem}
  \label{thm:nonlinear-control-final}
  Let $b > 3/2$ and $b_g > b + 3$. Let $f_{\stat}$ be a
  stationary state which is linearly stable in the sense of
  Definition \ref{def:linear-stability} and with velocity distribution
  \(g\) satisfying
  \begin{equation*}
    \| \hat{g} \|_{p_{b_g}} < \infty.
  \end{equation*}
  Furthermore, assume that one of the following conditions holds:
  \begin{itemize}
  \item $b > 3$,
  \item $b_g > b+5$.
  \end{itemize}
  Then there exist constants $C$ and $\delta$ such that for initial
  data $f_{\init}$ with the same velocity distribution \(g\) and
  \begin{equation*}
    \| \hat{f}_{\init} - \hat{f}_{\stat} \|_{p_b} \le \delta
  \end{equation*}
  there exists a unique global solution $f$ to
  \eqref{eq:kuramoto-mean-field-pde} and
  $\Theta : \R^+ \mapsto \R$ with
  \begin{equation*}
    |\Theta(0)| \le C\, \| \hat{f}_{\init} - \hat{f}_{\stat} \|_{p_b}
    \quad\text{and}\quad
    \left|\frac{\dd}{\dd t}{\Theta}(t)\right| \le C\, (1+t)^{2 - 2b} \| \hat{f}_{\init} - \hat{f}_{\stat} \|_{p_b}^2
  \end{equation*}
  such that the remaining perturbation
  \begin{equation*}
    u = R_{-\Theta(t)} \hat{f} - \hat{f}_{\stat}
  \end{equation*}
  has order parameter
  \begin{equation*}
    \conj{\eta}(t) = u_1(t,0)
  \end{equation*}
  decaying for all times $t$ as
  \begin{equation*}
    |\eta(t)| \le C\, (1+t)^{\frac 12 -b} \| u_{\init} \|_{p_b}.
  \end{equation*}
\end{theorem}

Measuring the decay through the order parameter $\eta$ of the
remaining perturbation $u$, this shows the decay of a small initial
perturbation. Furthermore, the bound on $\dot{\Theta}$ shows that the
system will converge to a nearby partially locked state because
$\Theta(t)$ converges to some $\Theta_\infty$ and
$|\Theta_\infty - \Theta(0)|$ is controlled by
$\| u_{\init} \|_{p_b}$. As the order parameter captures the
nonlinear behaviour further results of weak convergence can easily be
deduced, cf.~\cite[Theorem~39]{dietert-2016-stability-bifurcation}.

The needed regularity and the achieved decay rate are coming from the
handling of the rotation symmetry.

\section{Overview and setup}

By the rotation symmetry we consider throughout a stationary state
\(f_\stat\) with order parameter \(r_\stat \in [0,1]\).

Under the Fourier transform the evolution
\eqref{eq:kuramoto-mean-field-pde} of \(f\) becomes
\begin{equation}
  \label{eq:kuramoto-mean-field-pde-fourier}
  \left\{
    \begin{aligned}
      &\partial_t \hat{f}_\ell(t,\xi) = \ell \partial_\xi \hat{f}_\ell(t,\xi)
      + \frac{K\ell}{2}
      \left(
        \conj{r(t)}\, \hat{f}_{\ell-1}(t,\xi)
        - r(t)\, \hat{f}_{\ell+1}(t,\xi)
      \right),\\
      &\conj{r(t)} = \hat{f}_{1}(t,0).
    \end{aligned}
  \right.
\end{equation}

The stability of the stationary state $f_{\stat}$ is studied by
considering a solution $f = f_{\stat} + f_{\pert}$ where $f_{\pert}$
is the perturbation. From \eqref{eq:kuramoto-mean-field-pde-fourier}
the evolution is given by
\begin{equation*}
  \partial_t \hat{f}_{\pert} = L \hat{f}_{\pert} + Q(\hat{f}_{\pert})
\end{equation*}
where $L=L_1+L_2$ with
\begin{align*}
  (L_1 \hat{f}_{\pert})_{\ell}
  &= \ell \partial_\xi (\hat{f}_{\pert})_{\ell}
  + \frac{K\ell}{2} \left(r_{\stat}\, (\hat{f}_{\pert})_{\ell-1} -
  r_{\stat}\, (\hat{f}_{\pert})_{\ell+1}\right), \\
  (L_2 \hat{f}_{\pert})_{\ell}
  &= \frac{K\ell}{2} \left((\hat{f}_{\pert})_1(0)\, (\hat{f}_{\stat})_{\ell-1} -
    \conj{(\hat{f}_{\pert})_1(0)}\,  (\hat{f}_{\stat})_{\ell+1}\right),\\
  (Q(\hat{f}_{\pert}))_{\ell}
  &= \frac{K\ell}{2} \left((\hat{f}_{\pert})_1(0)\, (\hat{f}_{\pert})_{\ell-1} -
    \conj{(\hat{f}_{\pert})_1(0)}\,  (\hat{f}_{\pert})_{\ell+1}\right).
\end{align*}
Here \(L_1\) corresponds to the transport of the perturbation under
the order parameter \(r_{\stat}\) of the stationary state. This
transport creates the phase-mixing and thus the decay by Landau
damping. The operator \(L_2\) is a finite-dimensional operator
capturing the effect of the order parameter of the perturbation onto
the stationary state. The linear stability condition prescribes that
this effect is decaying. The nonlinearity \(Q\) is the effect of the
order parameter of the perturbation onto the transport of the
perturbation.

In Fourier space the rotation symmetry $R_{\Theta}$ acts as
\begin{equation}
  \label{eq:definition-hat-rot-theta}
  (\hat{R}_{\Theta}\hat{f})_\ell(\xi) = \ee^{\ii \ell \Theta} \hat{f}_\ell(\xi).
\end{equation}
The symmetry means that for any $\Theta$ and solution $f$ also
$R_{\Theta} f$ is a solution. In particular, $R_{\Theta} f_{\stat}$ is
also a stationary solution with the same behaviour. Therefore, along
the rotation symmetry, a perturbation does not decay and we need to
study orbital stability, i.e.\ if a solution $f$ converges to the set
$\{R_{\Theta} f_{\stat}\}_{\Theta\in\T}$.

In order to separate the rotation behaviour, we introduce polar type
coordinates for states close to the circle
$\{R_{\Theta} f_{\stat}\}_{\Theta\in\T}$. In these coordinates, the
solution is written as
\begin{equation*}
  \hat{f}(t) = \hat{R}_{\Theta(t)} (\hat{f}_{\stat} + u(t)),
\end{equation*}
where $\Theta(t)$ is a suitable chosen angle and $u$ is the remaining
perturbation. The time evolution of $u$ is then given by
\begin{equation}
  \label{eq:nonlinear-full-u}
  \partial_t u = L u + Q(u)
  - \frac{\dd \Theta}{\dd t}
  \left(
    \DD \hat{R} \hat{f}_{\stat}
    + \DD \hat{R} u
  \right),
\end{equation}
where $\DD \hat{R}$ denotes the differential of
$\Phi \to \hat{R}_\Phi$ at $\Phi = 0$.

Under the linear evolution and constant \(\Theta\) we find by
Duhamel's principle that
\begin{equation}
  \label{eq:duhamel-linear}
  u(t) = \ee^{tL_1} u_\init + \int_0^t \ee^{(t-s)L_1} L_2 u(s)\, \dd s.
\end{equation}
The operator \(L_2\) is the effect of the order-parameter of the
perturbation and can thus be written as
\begin{equation*}
  L_2 u = r_r\, \Re \conj{\eta}(t) + r_i\, \Im \conj{\eta}(t)
\end{equation*}
with
\begin{equation*}
  (r_r)_{\ell}(\xi) = \frac{K\ell}{2}
  \left(
    (\hat{f}_\stat)_{\ell-1}(\xi)
    - (\hat{f}_\stat)_{\ell+1}(\xi)
  \right)
\end{equation*}
and
\begin{equation*}
  (r_i)_{\ell}(\xi) = \frac{K\ell\ii}{2}
  \left(
    (\hat{f}_\stat)_{\ell-1}(\xi)
    + (\hat{f}_\stat)_{\ell+1}(\xi)
  \right).
\end{equation*}
Computing the order parameter \(\conj{\eta}_L(t)=u_1(t,0)\) from the
Duhamel formulation \eqref{eq:duhamel-linear} of the linearised
evolution, we find the Volterra convolution equation
\begin{equation}
  \label{eq:volterrra-eta-linearised}
  \begin{pmatrix}
    \Re \conj{\eta}_L \\ \Im \conj{\eta}_L
  \end{pmatrix}(t)
  +
  \left[
  k_{Lc} \conv
  \begin{pmatrix}
    \Re \conj{\eta}_L \\ \Im \conj{\eta}_L
  \end{pmatrix}
  \right](t)
  = \compF_L(t),
\end{equation}
with the usual convolution \(\conv\) denoting
\begin{equation*}
  \left[k_{Lc} \conv
  \begin{pmatrix}
    \Re \conj{\eta}_L \\ \Im \conj{\eta}_L
  \end{pmatrix}
  \right](t)
  =
  \int_0^t k_{Lc}(t-s)
  \begin{pmatrix}
    \Re \conj{\eta}_L \\ \Im \conj{\eta}_L
  \end{pmatrix}(s)
  \,\dd s
\end{equation*}
and the forcing
\begin{equation*}
  \compF_L(t) =
  \begin{pmatrix}
    \Re F_L(t) \\
    \Im F_L(t)
  \end{pmatrix}
  \qquad\text{with}\qquad
  F_L(t) =
  (\ee^{tL_1} u_\init)_1(0).
\end{equation*}
and the convolution kernel
\begin{equation*} \setlength\arraycolsep{4pt}
  k_{Lc}(\tau) = -
  \begin{pmatrix}
    \Re (\ee^{\tau L_1} r_r)_1(0) & \Re (\ee^{\tau L_1} r_i)_1(0) \\
    \Im (\ee^{\tau L_1} r_r)_1(0) & \Im (\ee^{\tau L_1} r_i)_1(0)
  \end{pmatrix}.
\end{equation*}
The unique solution of such a Volterra convolution equation
\eqref{eq:volterrra-eta-linearised} can be expressed by
\begin{equation}
  \label{eq:linear-volterra-resolvent-solution}
  \begin{pmatrix}
    \Re \conj{\eta}_L \\ \Im \conj{\eta}_L
  \end{pmatrix}
  = \compF_L - r_{Lc} \conv \compF_{L},
\end{equation}
where \(r_{Lc} : \R^+ \mapsto \C^2\) is the resolvent kernel, which is
the unique solution to
\begin{equation}
  \label{eq:linear-resolvent-definition}
  r_{Lc} + k_{Lc} \conv r_{Lc}
  = r_{Lc} + r_{Lc} \conv k_{Lc}
  = k_{Lc},
\end{equation}
see \cite[Theorem~3.1 and Theorem~3.5 of
Chapter~2]{gripenberg-londen-staffans-1990-volterra}.

The operator \(L_2\) is due to the explicit complex and real part of
\(\conj{\eta}\) not complex linear but only real linear. However,
treating \(\Re\conj{\eta}_L\) and \(\Im\conj{\eta}_L\) as formally
independent, we can perform a spectral analysis. As the convolution
turns into a product under the Laplace transform, the linear stability
can be determined by the characteristic equation
\begin{equation}
  \label{eq:characterstic-eq-linear-volterra}
  \det\Big(\Id + (\Lap k_{Lc})(z)\Big) = 0
\end{equation}
with the Laplace transform
\begin{equation*}
  (\Lap k_{Lc})(z)
  = \int_0^\infty k_{Lc}(t)\, \ee^{-zt}\, \dd t.
\end{equation*}

For the assumed regularity of the stationary state, \(k_{Lc}\) is
decaying with a sufficiently fast polynomial rate to define the linear
stability:
\begin{definition}
  \label{def:linear-stability}
  The stationary state $f_{\stat}$ is linearly stable up to the
  rotation invariance if $z=0$ is the only solution of the
  characteristic equation \eqref{eq:characterstic-eq-linear-volterra}
  in $\Re z \ge 0$ and
  \begin{equation*}
    \frac{\dd}{\dd z}
    \det\Big(\Id + (\Lap k_{Lc})(z)\Big)
    \Bigg|_{z=0}
    \not = 0.
  \end{equation*}
\end{definition}
Indeed we will see in Proposition~\ref{thm:linear-resolvent-bound}
that this condition is sharp for the original system and no spurious
eigenmodes are created by the complexification. Moreover, up to a
change of basis corresponding to a different complexification this
agrees with the linear stability condition from the analytic setting
in \cite{dietert-fernandez-gerard-varet-2018-landau-damping-pls}.

With the rotation symmetry the resolvent takes the form
\begin{equation}
  \label{eq:linear-resolvent-splitting}
  r_{Lc} = K_{\Theta} + r_{Lcs},
\end{equation}
where \(r_{Lcs}\) is decaying by the linear stability and $K_\Theta$
is a constant matrix corresponding to the rotation. Looking at the
rotation eigenmode, we find that $K_{\Theta}$ takes the form
\begin{equation}
  \label{eq:volterra-rotation-eigenmode-resolvent-form}
  K_{\Theta} =
  \begin{pmatrix}
    0 \\ 1
  \end{pmatrix}
  \otimes
  \begin{pmatrix}
    c_{\Theta,r} & c_{\Theta,i}
  \end{pmatrix},
\end{equation}
where \(c_{\Theta,r}\) and \(c_{\Theta,i}\) are constants.

Hence the order parameter \(\conj{\eta}_L\) is decaying if and only if
the forcing $F_L$ satisfies
\begin{equation*}
  K_{\Theta}
  \int_0^\infty
  \compF_L(t)
  \,\dd t
  = 0
  \qquad\text{or}\qquad
  \int_0^\infty \Big(c_{\Theta,r} \Re F_L(t) + c_{\Theta,i} \Im F_L(t)\Big)\, \dd t = 0.
\end{equation*}
Therefore, we characterise the stable part of the linear evolution by
the functional \(\alpha\) given by
\begin{equation}
  \label{eq:stable-subspace-functional-alpha}
  \alpha(u)
  = \int_0^\infty \Big(c_{\Theta,r} \Re (\ee^{tL_1} u)_1(0) + c_{\Theta,i} \Im (\ee^{tL_1} u)_1(0)\Big)\, \dd t.
\end{equation}

For the full nonlinear behaviour, we therefore want to choose the
angle \(\Theta(t)\) such that always \(\alpha(u(t)) = 0\). Initially,
we can find such an angle \(\Theta(0)\), see
Lemma~\ref{thm:local-choice-theta}, and, as long as the perturbation
is small enough, this condition is propagated in
\eqref{eq:nonlinear-full-u} if
\begin{equation*}
  \frac{\dd \Theta}{\dd t}
  = \dot{\Theta},
\end{equation*}
where
\begin{equation}
  \label{eq:def-theta-dot}
  \dot{\Theta} = \dot{\Theta}(u)
  := \frac{\alpha(Q(u))}{\alpha(\DD\hat{R}\hat{f}_{\stat})+\alpha(\DD\hat{R}u)}.
\end{equation}
As $\dot{\Theta}$ is a function of $u$, this defines a closed
evolution for the perturbation \(u\).

For proving the nonlinear stability result, the challenge is that the
nonlinear terms are unbounded and cannot be controlled by the linear
evolution as perturbation. However, they have the same structure as
\(L_1\) for which we have a good energy estimate for the stability.

In order to exploit the structure we perform a bootstrap argument:
taking \(\eta(t) = \conj{u_1(t,0)}\) and
\(\dot{\Theta}(t) = \dot{\Theta}(u(t))\) as given functions of time, we
define the time-dependent complex linear operator
$B_{1n}^{\eta,\dot{\Theta}}$ and time-dependent real linear operator
$B_2^{\eta,\dot{\Theta}}$ by
\begin{align}
  \label{eq:def-b1n}
  (B^{\eta,\dot{\Theta}}_{1n} v)_\ell
  &= \frac{K\ell}{2}
  \Big(\conj{\eta(t)}\, v_{\ell-1} - \eta(t)\, v_{\ell+1}\Big)
  - \ii \ell \dot{\Theta}(t)\, v_{\ell}, \\
  (B^{\eta,\dot{\Theta}}_{2} v)_\ell
  &= - \alpha(B^{\eta,\dot{\Theta}}_{1n} v)\, r_{\Theta},
\end{align}
where
\begin{equation*}
  r_{\Theta} := \frac{\DD \hat{R} \hat{f}_{\stat}}
  {\alpha(\DD \hat{R} \hat{f}_{\stat})}.
\end{equation*}

For the rotation we find
\begin{equation*}
  (\DD \hat{R} u)_\ell = \ii \ell u_\ell
\end{equation*}
so that the solution \(u\) of the full evolution satisfies
\begin{equation*}
  B^{\eta,\dot{\Theta}}_{1n} u = Q(u) - \dot{\Theta}\, \DD \hat{R} u
\end{equation*}
and
\begin{equation*}
  B^{\eta,\dot{\Theta}}_{2} u = - \dot{\Theta}\, \DD \hat{R} f_\stat.
\end{equation*}

Hence the solution can be expressed as
\begin{equation}
  \label{eq:solution-u-linear-b}
  \partial_t u = B^{\eta,\dot{\Theta}} u + L_2 u
\end{equation}
with
\(B^{\eta,\dot{\Theta}} = B^{\eta,\dot{\Theta}}_1 +
B^{\eta,\dot{\Theta}}_2\) and
\(B^{\eta,\dot{\Theta}}_1 = L_1 + B^{\eta,\dot{\Theta}}_{1n}\).  Here
we collected terms with the same structure as \(L_1\) in
\(B^{\eta,\dot{\Theta}}_1\) and \(B^{\eta,\dot{\Theta}}_2\) is a
finite-dimensional operator keeping \(\alpha(u) = 0\).

For better readability, we will often drop the explicit dependence of
\(\eta,\dot{\Theta}\) and just write \(B_1=B^{\eta,\dot{\Theta}}_1\),
\(B_2=B^{\eta,\dot{\Theta}}_2\) and \(B=B^{\eta,\dot{\Theta}}\).

For the (time-dependent) operators \(E=L_1,B^{\eta,\dot{\Theta}}_{1}\)
or \(B^{\eta,\dot{\Theta}}\), the evolution problem
\begin{equation*}
  \partial_t v = E(t)\, v
\end{equation*}
has a unique weak-solution over a time range where the coefficients
are continuous. The corresponding solution operator from time \(s\) to
time \(t\) is denoted by \(\sol{E}st\), i.e.\
\(v(t) = \sol{E}st v_{\init}\) is defined via
\begin{equation*}
  \partial_t v(t) = E(t)\, v(t),\quad v(s) = v_{\init}.
\end{equation*}

Hence the solution $u$ can be expressed by Duhamel's principle as
\begin{equation}
  \label{eq:solution-u-duhamel}
  u(t) = \sol{B}0t u_{\init} + \int_0^t \sol{B}st L_2 u(s)\, \dd s
\end{equation}
with the initial data $u_{\init}$ at time $t=0$. Then the real linearity of \(B\) implies that
\begin{equation*}
  \sol Bst (L_2 u(s))
  = (\sol Bst r_r)\, (\Re \conj{\eta}(t))
  + (\sol Bst r_i)\, (\Im \conj{\eta}(t)).
\end{equation*}
Computing the order parameter $\conj{\eta}(t) = u_1(t,0)$ over
\eqref{eq:solution-u-duhamel}, we find the Volterra equation
\begin{equation}
  \label{eq:volterrra-eta}
  \begin{pmatrix}
    \Re \conj{\eta} \\ \Im \conj{\eta}
  \end{pmatrix}
  (t)
  +
  \left[
    k \conv
    \begin{pmatrix}
      \Re \conj{\eta} \\ \Im \conj{\eta}
    \end{pmatrix}
  \right](t) = \compF(t)
\end{equation}
with the Volterra kernel
\begin{equation*} \setlength\arraycolsep{4pt}
  k(t,s) = -
  \begin{pmatrix}
    \Re (\sol Bst r_r)_1(0) & \Re (\sol Bst r_i)_1(0) \\
    \Im (\sol Bst r_r)_1(0) & \Im (\sol Bst r_i)_1(0)
  \end{pmatrix}
\end{equation*}
using the product notation
\begin{equation}
  \label{eq:volterra-special-product}
  \left[k \conv
  \begin{pmatrix}
    \Re \conj{\eta} \\ \Im \conj{\eta}
  \end{pmatrix}
  \right](t)
  =
  \int_0^t k(t,s)
  \begin{pmatrix}
    \Re \conj{\eta} \\ \Im \conj{\eta}
  \end{pmatrix}(s)
  \,\dd s.
\end{equation}
and the forcing
\begin{equation*}
  \compF(t) =
  \begin{pmatrix}
    \Re F(t) \\
    \Im F(t)
  \end{pmatrix}
  \qquad\text{with}\qquad
  F(t) = (\sol B0t u_\init)_1(0).
\end{equation*}

The linear convolution kernel \(k_{Lc}\) can be identified in this
general Volterra equation by \(k_{L}(t,s) = k_{Lc}(t-s)\) for
\(t \ge s\). Assuming by a bootstrap assumption that \(\eta(t)\) and
\(\dot{\Theta}\) are small, we can control the difference between the
Volterra kernel \(k\) and the Volterra kernel of the linearised
evolution \(k_L\). Here we use that for the kernel elements, we always
evolve \(r_r\) and \(r_i\), which are constant in time and can have
better regularity. For a small enough difference between \(k\) and
\(k_{L}\), we can find a resolvent for the Volterra equation by a von
Neumann series. This in turn then controls the order parameter.

Hence we can sketch the overall bootstrap argument: assuming that
\(\eta(t)\) and \(\dot{\Theta}\) are small, we show by the Volterra
equation that for a perturbation \(v\) the order parameter
\((\sol{B}st v)_1(0)\) is decaying. With the control of the order
parameter, we can go back to \eqref{eq:solution-u-duhamel} and control
seminorms of the perturbation, which then close the bootstrap argument.

Here the key point is that we handle the nonlinearities not as forcing
but by the time-dependent transport \((\sol{B}st v)_1(0)\) of the
perturbations \(v\) and this allows the exploitation of the divergence
structure of the nonlinearity. This means that we arrive at a Volterra
equation with a modified kernel, where we can exploit the different
regularity of the different parts of the perturbation.

As noted before the velocity distribution \(g\) is constant in time
and we assumed that the initial data \(f_{\init}\) and the nearby
stationary state \(f_{\stat}\) have the same velocity
distribution. This means that for the perturbation \(u_0 \equiv
0\). In \cite{dietert-2016-stability-bifurcation} it was noted that
only the nearest Fourier modes interact so that one can consider the
evolution of the perturbation in \(\ell \ge 1\) and \(\xi \ge 0\).

For the different terms of the perturbation \(u\), we therefore
introduce the weighted Sobolev spaces
\begin{equation*}
  \sX_{\phi,k} = \{v : \N \times \R_{\ge 0} \mapsto \C \text{ with }
  \| v \|_{\phi,k} < \infty\}
\end{equation*}
with the norm
\begin{equation*}
  \| v \|_{\phi,k}^2
  = \sum_{\ell \ge 1} \int_0^\infty
  \Big(
    |v_\ell(\xi)|^2 + |\partial_\xi v_\ell(\xi)|^2
  \Big)
  |\phi(\xi)|^2 \ell^{2k} \dd \xi
\end{equation*}
for the weight $\phi$ and degree $k$ and use the shorthands
\begin{equation*}
  \sX_{\phi} = \sX_{\phi,-\frac 12}
  \quad\text{and}\quad
  \| v \|_{\phi} = \| v \|_{\phi,-\frac 12}.
\end{equation*}
For the weight \(\phi\) we use \(p_{A,b}\) defined by
\begin{equation}
  \label{eq:def-weight-pab}
  p_{A,b}(\xi) = (A+\xi)^b
  \qquad\text{ with  }\qquad
  p_{b} = p_{1,b}.
\end{equation}
For \(v \in \sX\) we use throughout in the definition of the operators
the convention that \(v_0 \equiv 0\).

The structure of the norm is similar to the used norms in
\cite{dietert-fernandez-gerard-varet-2018-landau-damping-pls}. The
weight \(\phi\) encodes the half-sided regularity, which is
transformed to decay of the order parameter under the
transport. Otherwise the norm is motivated by the fact that we want
pointwise control and a Hilbert space structure.

Recalling the analogous definition \eqref{eq:def-norm-ghat} of
\(\| \hat{g} \|_{p_b}\), the first result is to show that the norms
are well-adapted to the stationary states $\hat{f}_{\stat}$, as
defined in \eqref{eq:stationary-state-physical}.
\begin{proposition}
  \label{thm:regularity-stationary-state}
  Let $b \ge 0$ and $k=0,1/2,1,3/2,\dots$. Then there exists a
  constant $C_\stat$ only depending on $k$, $b$ and $K r_\stat$ such
  that the stationary state \(f_{\stat}\) satisfies
  \begin{equation*}
    \| \hat{f}_{\stat} \|_{p_{b},k}
    \le C_{\stat} \| \hat{g} \|_{p_{b+k+1}},
  \end{equation*}
  where $\hat{f}_{\stat}$ is restricted to $\ell \ge 1$ and
  $\xi \ge 0$.
\end{proposition}
A crucial ingredient for the control is that we have chosen our
stationary state $\hat{f}_{\stat}$ in
\eqref{eq:stationary-state-physical} such that all locked oscillators
are at the stable fixed-point, cf.\
\cite{dietert-fernandez-gerard-varet-2018-landau-damping-pls}.

In the spaces \(\sX_{p_b}\) we can show that the transport under
\(L_1\) and \(B_1\) is well-defined and quantify the decay through the
transport.

\begin{lemma}
  \label{thm:evolution-l1-b1}
  Let $E=L_1$, or let $E = B_1$ and assume that
  the coefficients of $B_1$ are continuous for the considered time
  range $J=[0,T]$. Fix $b \ge 0$. Then for \(v \in \sX_{p_b}\) and
  \(s \in J\) the evolution equation
  \begin{equation*}
    \left\{
      \begin{aligned}
        &\partial_t w = E\, w,\\
        &w(s) = v
      \end{aligned}
    \right.
  \end{equation*}
  has a unique weak solution $w \in C_w([s,T], \sX_{p_b})$, i.e.\
  $w \in L^\infty([s,T],\sX_{p_{b}})$ and is weakly
  continuous. Setting \(\sol{E}st v = w(t)\) defines a well-defined
  solution operator, which is complex linear. For \(s\le t\le T\) and
  \(A\ge 1\) the following estimate holds
  \begin{equation*}
    \| \sol Est v \|^2_{p_{A+t-s,b}}
    +
    \int_s^t
    \left| \left(\sol Es\tau v\right)_1(0) \right|^2 (A+\tau-s)^{2b}\, \dd \tau
    \le
    \| v \|^2_{p_{A,b}}.
  \end{equation*}
\end{lemma}

For \(L_1\) the coefficients are constant and we have just constructed
the semigroup \(\ee^{tL_1}\) as \(\sol{L_1}st = \ee^{(t-s)L_1}\).

For the effect of a perturbation part \(v\) we introduce the seminorms
$\beta_{\eta}$, $\beta_{\alpha}$ and $\beta_{d}$ by
\begin{align}
  \label{eq:def-seminorm-beta}
  \beta_{\eta}(v) &= |v_1(0)|, \\
  \label{eq:def-seminorm-ba}
  \beta_{\alpha}(v) &=  \int_0^\infty
                      \left| \left(\ee^{tL_1}v\right)_1(0) \right|
                      \,\dd t, \\
  \label{eq:def-seminorm-bd}
  \beta_d(v) &= \max
               \left\{
               \beta_{\alpha}\Big( (\ell v_{\ell-1})_{\ell} \Big),
               \beta_{\alpha}\Big( (\ell v_\ell)_{\ell} \Big),
               \beta_{\alpha}\Big( (\ell v_{\ell+1})_{\ell} \Big)
               \right\}.
\end{align}
Here \(\beta_{\eta}\) allows a control of the order parameter, while
\(\beta_{\alpha}(v)\) will allow a control of \(\alpha(v)\). Finally
\(\beta_d(v)\) will allow a control of
\(\alpha(B^{\eta,\dot{\Theta}}_{1n}v)\).

These seminorms can be controlled by the weighted Sobolev norms as
follow.
\begin{lemma}
  \label{thm:control-seminorms-by-weighted-sobolev}
  There exists a numerical constant $C_S$ such that for $A \ge 1$ and
  $b \ge 0$
  \begin{equation*}
    \beta_\eta(u) \le C_S\, A^{-b} \, \| u \|_{p_{A,b}}.
  \end{equation*}
  If $b > 1/2$ it holds that
  \begin{equation*}
    \beta_{\alpha}(u) \le \frac{A^{\frac 12 - b}}{\sqrt{2b-1}}
    \| u \|_{p_{A,b}}.
  \end{equation*}
  For $b > 3/2$ there exists a constant $C_{\beta d}$ only depending
  on $b$ and $K r_{\stat}$ such that for $A \ge 1$
  \begin{equation*}
    \beta_d(u) \le C_{\beta d}\, A^{\frac 32 -b} \,
    \| u \|_{p_{A,b}}.
  \end{equation*}
\end{lemma}

Hence we see how the order parameter decays in time under the
transport of \(L_1\) and \(B_1\). In order to quantify the decays in
time we introduce the weighted norms
\begin{align*}
  \| h \|_{L^1(J,\phi)}
  &= \int_{t \in J} \| h(t) \|\, \phi(t)\, \dd t,\\
  \| h \|_{L^\infty(J,\phi)}
  &= \esssup_{t \in J} \| h(t) \|\, \phi(t),
\end{align*}
for a function \(h\) of time over a considered time range $J = [0,T]$
or $J = \R^+$. Here we will again use the submultiplicative weight
functions $p_{A,b}$ from \eqref{eq:def-weight-pab} as weights
\(\phi\).

By the estimate of Lemma \ref{thm:evolution-l1-b1}, we can quantify
the decay of the convolution kernel \(k_{Lc}\):
\begin{lemma}
  \label{thm:decay-linear-volterra-kernel}
  Let $0 \le b \le b_r - 1/2$. Then there exists a numerical constant
  $C$ such that
  \begin{equation*}
    \int_0^\infty \| k_{Lc}(t) \|\, (1+t)^b\, \dd t
    \le C
    \Big(
    \| r_r \|_{p_{b_r}}
    + \| r_i \|_{p_{b_r}}
    \Big).
  \end{equation*}
\end{lemma}

If $\|r_{r}\|_{p_{b_r}} < \infty$ and $\|r_{i}\|_{p_{b_r}} < \infty$
for $b_r > 1/2$, this shows that the Laplace transform in
\eqref{eq:characterstic-eq-linear-volterra} is defined for
$\Re z \ge 0$ by an absolutely converging integral. Moreover, if
$b_r > 3/2$, then the Laplace transform $\Lap k_{Lc}$ is continuous
differentiable in the whole region $\{z \in \C: \Re z \ge 0\}$, in
particular, including the critical line $\Re z = 0$.

Hence under suitable regularity of the stationary state, the linear
stability condition from Definition~\ref{def:linear-stability} is
well-defined and the linear stability is quantified by the decay of
the stable part \(r_{Lcs}\) of the resolvent in \(r_{Lc}\) in
\eqref{eq:linear-resolvent-splitting} as
\begin{equation*}
  \| r_{Lcs} \|_{L^1(\R^+,p_b)} < \infty
\end{equation*}
for a parameter $b \ge 0$.

As the forcing \(F_L\) is decaying for regular initial data by
Lemma~\ref{thm:evolution-l1-b1}, this will imply a decay of the order
parameter \(\conj{\eta}_L\) for initial data \(u_\init\) with
\(\alpha(u_\init)=0\) under the linear evolution. Here the effect of
the \(K_\Theta\) is that for a perturbation initially in \(\sX_{p_b}\)
the order parameter only decays as \((1+t)^{\frac 12 -b}\) instead of
the rate \((1+t)^{-b}\) under the linear transport alone, cf.\
Corollary~\ref{thm:forcing-control-eigenmode-b-alpha}. This is the
reason for the obtained rate in the main
Theorem~\ref{thm:nonlinear-control-final}.

For treating the full evolution, we assume for the coefficients
\(\eta\) and \(\dot{\Theta}\) in \(B^{\eta,\dot{\Theta}}_1\) and
\(B^{\eta,\dot{\Theta}}_2\) the bootstrap control
\begin{equation}
  \label{eq:bootstrap-def-d}
  R_d(t) =
  \sup_{s \in [0,t]} (1+s)^{b_d} M_d(s)
  \qquad\text{ with }\qquad
  M_d(t) = K |\eta(t)| + |\dot{\Theta}(t)|
\end{equation}
with a decay rate $b_d \ge 0$.

With the definition \eqref{eq:stable-subspace-functional-alpha} of
\(\alpha\) and the bootstrap assumption \eqref{eq:bootstrap-def-d}, we
can control the evolution under \(B\). Here we need to control
\(B_2\), which implies with the definition of \(\alpha\) the required
regularity of the perturbation in the main
Theorem~\ref{thm:nonlinear-control-final}. We suspect that it is
optimal because for the homogeneous state (\(r_{\stat}=0\)) we can
solve the integral in the definition of \(\alpha\) explicitly and find
that the rate is sharp.

\begin{lemma}
  \label{thm:time-evolution-b}
  Assume continuous coefficients \(\eta\) and \(\dot{\Theta}\) for
  \(B^{\eta,\dot{\Theta}}\) over the time range $J=[0,T]$ and that
  \begin{equation*}
    \| r_{\Theta} \|_{p_{b_r}} < \infty
  \end{equation*}
  for some $b_{r} > 3/2$. Fix \(b\) such that \(b_r \ge b >
  3/2\). Then for \(v \in \sX_{p_b}\) and \(s \in J\) the evolution
  equation
  \begin{equation*}
    \left\{
      \begin{aligned}
        &\partial_t w = B\, w,\\
        &w(s) = v
      \end{aligned}
    \right.
  \end{equation*}
  has a unique weak solution $w \in C_w([s,T], \sX_{p_b})$ and the
  corresponding solution operator \(\sol{B}st\) is well-defined and
  real linear.

  With the control \eqref{eq:bootstrap-def-d} of the coefficients, it
  holds for $v \in \sX_{p_b}$ and $0 \le s \le t \le T$ that
  \begin{align*}
    \beta_d(\sol Bst v)
    &\le C_{\beta d} (1+t-s)^{\frac 32 -b}
    \| v \|_{p_{b}} \\
    &+ C_{\beta d} \| r_{\Theta} \|_{p_{b_r}} \|K_{\Theta}\| R_d(t)
    \int_s^t (1+t-\tau)^{\frac 32 - b_r} (1+\tau)^{-b_d}
    \beta_d( \sol Bs\tau v )\, \dd \tau,\\
    \beta_\alpha(\sol Bst v)
    &\le \frac{(1+t-s)^{\frac 12 -b}}{\sqrt{2b-1}}
    \| v \|_{p_{b}} \\
    &+ \frac{\| r_{\Theta} \|_{p_{b_r}}\|K_{\Theta}\|}{\sqrt{2b_r-1}} R_d(t)
    \int_s^t (1+t-\tau)^{\frac 12 - b_r} (1+\tau)^{-b_d}
    \beta_d( \sol Bs\tau v )\, \dd \tau,\\
    \beta_\eta(\sol Bst v)
    &\le C_S (1+t-s)^{-b}
    \| v \|_{p_{b}} \\
    &+ C_{S} \| r_{\Theta} \|_{p_{b_r}} \|K_{\Theta}\| R_d(t)
    \int_s^t (1+t-\tau)^{- b_r} (1+\tau)^{-b_d}
    \beta_d( \sol Bs\tau v )\, \dd \tau.
  \end{align*}

  If additionally $b_r > b+1$ or $b_d > 1$, then there exists a
  constant $\delta_R$ such that
  \begin{equation*}
    \beta_d(\sol Bst v)
    \le 2 C_{\beta d} (1+t-s)^{\frac 32 -b}
    \| v \|_{p_b}
  \end{equation*}
  if $R_d(t) \le \delta_R$.
\end{lemma}

This then shows that under the bootstrap assumption the solution can
indeed be written as \eqref{eq:solution-u-duhamel}.

For the resulting Volterra equation \eqref{eq:volterrra-eta} of the
full evolution and a considered time range $J=[0,T]$, the Volterra
kernel can be extended to a function on $J \times J$ by setting
$k(t,s) = 0$ for $t<s$. Following Section~2 of Chapter~9 of
\cite{gripenberg-londen-staffans-1990-volterra}, a suitable norm for a
kernel $k$ is
\begin{equation}
  \label{eq:def-general-kernel-norm}
  \tnorm{k}_{L^\infty(J,\phi)} :=
  \sup_{t \in J} \phi(t) \int_J \| k(t,s) \| (\phi(s))^{-1} \dd s,
\end{equation}
where $\phi$ is a submultiplicative weight (i.e.\
$\phi(t+s) \le \phi(t) \phi(s)$) and we let $\vkV(J,\phi)$ be the
class of such functions. As in \eqref{eq:volterra-special-product},
define the general convolution product between $k \in \vkV(J,\phi)$
and a function $F$ on $J$ as
\begin{equation}
  \label{eq:def-general-product-forcing}
  (k \conv F)(t) = \int_0^t k(t,s) F(s)\, \dd s
\end{equation}
and for $\beta,\gamma \in \vkV(J,\phi)$ as
\begin{equation}
  \label{eq:def-general-product}
  (\beta \conv \gamma)(t,s) =
  \int_{\tau =s}^t \beta(t,\tau) \gamma(\tau,s)\, \dd \tau.
\end{equation}
This gives a Banach algebra structure and allows to handle the
nonlinear behaviour by a von Neumann series. This then gives the
following nonlinear control of the order parameter in terms of the
forcing.
\begin{lemma}
  \label{thm:control-eta-volterra}
  Let $b_{\eta} \ge 0$, $b_d \ge 0$ and $b_r > b_{\eta}+ 5/2$ with
  \begin{equation*}
    \| r_{r} \|_{p_{b_r}} < \infty,\qquad
    \| r_{i} \|_{p_{b_r}} < \infty,\qquad
    \| r_{\Theta} \|_{p_{b_r}} < \infty,
  \end{equation*}
  and
  \begin{equation*}
    \| r_{r} \|_{p_{b_r-\frac 12},0} < \infty,\qquad
    \| r_{i} \|_{p_{b_r-\frac 12},0} < \infty,\qquad
    \| r_{\Theta} \|_{p_{b_r-\frac 12},0} < \infty.
  \end{equation*}
  Assume that one of the following conditions holds:
  \begin{itemize}
  \item $b_d > 5/2$,
  \item $b_r > b_{\eta} + 4 + \max\{0,1-b_d\}$ and
    ${b_r > b_{\eta} + \frac{17}{4} + \max\{0,1-b_d\} +
    \frac{\max\left\{0,\frac 32 - b_d\right\}}{2} - \frac{b_d}{2}}$.
  \end{itemize}

  Furthermore, assume that the stationary state $f_{\stat}$ is
  linearly stable in the sense of Definition \ref{def:linear-stability}. Then
  there exist constants $C_\eta$ and $\delta_R$ such that, for a
  forcing
  \begin{equation*}
    \compF(t) =
    \begin{pmatrix}
      \Re F(t) \\
      \Im F(t)
    \end{pmatrix}
    \qquad\text{with}\qquad
    F(t) = (\sol B0t u_\init)_1(0)
  \end{equation*}
  with $\| u_{\init} \|_{p_{b}} < \infty$ for $b > 3/2$ and
  $\alpha(u_{\init}) = 0$, the solution $\Re \conj{\eta}(t)$ and
  $\Im \conj{\eta}(t)$ of the Volterra equation
  \eqref{eq:solution-u-duhamel} is controlled by
  \begin{equation*}
    |\eta(t)| \le C_{\eta} (1+t)^{-b_{\eta}}
    \sup_{s \in [0,t]}
    \Big(
      |F(s)| + \beta_{\alpha}\left(\sol B0s u_{\init}\right)
    \Big)
    (1+s)^{b_{\eta}}
  \end{equation*}
  if $R_d(t) \le \delta_R$ and the coefficients \(\eta\) and
  \(\dot{\Theta}\) for \(B^{\eta,\dot{\Theta}}\) are continuous up to
  time \(t\).
\end{lemma}

Having understood the different parts, we can now assemble the
bootstrap argument. First, by combining
Lemma~\ref{thm:control-eta-volterra} with the control of the forcing,
we control the order parameter under the bootstrap hypothesis.
\begin{lemma}
  \label{thm:bootstrap-order-parameter}
  Let $b > 3/2$, $b_d = b - 1/2$ and $b_r > b + 3/2$. Let
  $\hat{f}_{\stat}$ be a linearly stable stationary state in the sense
  of Definition~\ref{def:linear-stability} such that
  \begin{equation*}
    \| r_{r} \|_{p_{b_r}} < \infty,\qquad
    \| r_{i} \|_{p_{b_r}} < \infty,\qquad
    \| r_{\Theta} \|_{p_{b_r}} < \infty,
  \end{equation*}
  and
  \begin{equation*}
    \| r_{r} \|_{p_{b_r-\frac 12},0} < \infty,\qquad
    \| r_{i} \|_{p_{b_r-\frac 12},0} < \infty,\qquad
    \| r_{\Theta} \|_{p_{b_r-\frac 12},0} < \infty.
  \end{equation*}
  Furthermore, assume that one of the following conditions holds:
  \begin{itemize}
  \item $b_d > 5/2$,
  \item $b_r > b + 7/2$.
  \end{itemize}
  Then there exist constants $\delta_R$ and $C$ such that
  \begin{equation*}
    |\eta(t)| \le C\, (1+t)^{- b_d}
    \| u_{\init} \|_{p_b}
  \end{equation*}
  if $R_d(t) \le \delta_R$ and the coefficients \(\eta\) and
  \(\dot{\Theta}\) for \(B^{\eta,\dot{\Theta}}\) are continuous up to
  time \(t\).
\end{lemma}
With the control of the order parameter, we can use
\eqref{eq:solution-u-duhamel} to control $\beta_d(u(t))$.
\begin{lemma}
  \label{thm:bootstrap-beta-d}
  Let $b > 3/2$, $b_d = b - 1/2$ and $b_r > b+1$ and assume that
  \begin{equation*}
    \|r_{\Theta}\|_{p_{b_r}} < \infty
    \quad\text{ and }\quad
    \|r_{r}\|_{p_{b_r}} < \infty
    \quad\text{ and }\quad
    \|r_{i}\|_{p_{b_r}} < \infty.
  \end{equation*}
  Then there exist constants $\delta_R$ and $C$ such that
  \begin{equation*}
    |\beta_d(u(t))| \le C
    (1+t)^{\frac 32 - b}
    \Bigg(
    \| u_{\init} \|_{p_b}
    + \sup_{s \in [0,t]} (1+s)^{b_d} |\eta(s)|
    \Bigg)
  \end{equation*}
  if $R_d(t) \le \delta_R$ and the coefficients \(\eta\) and
  \(\dot{\Theta}\) for \(B^{\eta,\dot{\Theta}}\) are continuous up to
  time \(t\).
\end{lemma}

By a well-posedness result of the nonlinear evolution, we can prove
that $\eta$ and $\dot{\Theta}$ vary continuously as long as
\begin{equation*}
  |\beta_d(u(t))| < \frac{1}{2} \, |\alpha(\DD\hat{R}f_{\stat})|.
\end{equation*}
In this case, we can also control
$\dot{\Theta}$ by
\begin{equation*}
  |\dot{\Theta}(t)| \le C\, |\eta(t)|\, \beta_d(u(t))
\end{equation*}
for a constant $C$. Combining the previous estimates we can therefore
prove the result by a bootstrap argument.

\section{Norms and time-evolution under the transport operators}
\label{sec:norms-time-evolution}

The bound on the stationary state comes from an energy estimate with
an appropriate approximation scheme for this class of partially locked
states.

\begin{proof}[Proof Proposition \ref{thm:regularity-stationary-state}]
  As $\hat{f}_{\stat}$ is a stationary state, by
  \eqref{eq:kuramoto-mean-field-pde-fourier} it satisfies
  \begin{equation*}
    0 = \ell \partial_\xi (\hat{f}_{\stat})_\ell(\xi)
    + \frac{K\ell}{2}
    \left(
      r_{\stat}\, (\hat{f}_{\stat})_{\ell-1}
      - r_{\stat}\, (\hat{f}_{\stat})_{\ell+1}
    \right)
  \end{equation*}
  and $(\hat{f}_{\stat})_0 = \hat{g}$. For an a priori estimate, let
  $b \ge 0$ and take the inner product in $\sX_{p_b,-\frac 12}$ with
  \(\hat{f}_{\stat}\). This gives
  \begin{equation*}
    2b \| \hat{f}_{\stat} \|_{p_{b-\frac 12},0}^2
    \le \frac{K r_{\stat}}{2} \| \hat{g} \|_{p_{b+\frac 12}}
    \| \hat{f}_{\stat} \|_{p_{b-\frac 12},0},
  \end{equation*}
  which shows the result for $k=0$.

  Fixing $k=1/2,1,3/2,\dots$, we find the a priori estimate by
  taking the inner product in $\sX_{p_b,k-\frac 12}$
  \begin{equation*}
    2b \| \hat{f}_{\stat} \|_{p_{b-\frac 12},k}^2
    \le C \| \hat{f}_{\stat} \|^2_{p_{b},k-\frac 12}
    + C \| \hat{f}_{\stat} \|_{p_{b},k-\frac 12} \| \hat{g} \|_{p_{b}},
  \end{equation*}
  for some constant \(C\) using that \(|r_{\stat}|\le 1\), which
  implies the result by induction over $k$.

  The a priori estimates are justified for states with all locked
  oscillators at the stable fixed points. For this construct
  approximate states $\hat{f}^n_{\stat}$ as Fourier transform of
  $f^n_{\stat}$ given by
  \begin{equation*}
    f^n_{\stat}(\theta,\omega) =
    \begin{dcases}
      \delta_{\arcsin(\omega/(Kr_{\stat}))}(\theta)\, g^n(\omega)
      &\text{if $|\omega|\le K r_{\stat}$}, \\
      \frac{\sqrt{\omega^2 - (Kr_{\stat})^2}}{2\pi |\omega -
        Kr_{\stat}\sin \theta|}\, g^n(\omega)
      &\text{if $|\omega|>K r_{\stat}$},
    \end{dcases}
  \end{equation*}
  where $g^n$ is an approximation of $g$ such that $g^n$ has analytic
  regularity and $\| \hat{g}^n - \hat{g} \|_{p_b} \to 0$, e.g.\ $g^n$
  is obtained by convolution of $g$ with a Gaussian. By
  \cite{dietert-fernandez-gerard-varet-2018-landau-damping-pls},
  we then control
  \begin{equation*}
    \int_0^\infty \Big( |\hat{f}^n_{\stat}|^2 + |\partial_\xi
    \hat{f}^n_{\stat}|^2 \Big)
    \, \ee^{2a \xi}\, \dd \xi
    \le C\, \delta^\ell
  \end{equation*}
  with $a > 0$, $\ell \in \N$ for constants $C$ and $\delta <
  1$. Hence for $\hat{f}^n_{\stat}$ the a priori estimates are
  justified. Since
  $(\hat{f}^n_{\stat})_{\ell}(\xi) \to (\hat{f}_{\stat})_\ell(\xi)$ as
  $n\to \infty$, this shows the claimed bound.
\end{proof}

The results on the evolution operators are based on energy
estimates. The derivatives $\partial_\xi u_\ell$ can always be handled
in the same way, because they satisfy the same evolution equation.

\begin{proof}[Proof of Lemma \ref{thm:evolution-l1-b1}]
  By Morray's inequality, a function $w \in C_w([s,T], \sX_{p_{b}})$
  is uniformly continuous in $\xi$ over compact regions. Moreover, by
  the weak continuity $w_\ell(\cdot,\xi)$ is continuous in time for
  all $\ell \in \N$ and $\xi \in \R^+$ and so $w$ is a continuous
  function. By standard arguments on the scalar transport equation,
  this shows the uniqueness of solutions.

  For constructing a solution we use the following a priori estimate
  for $B_1$
  \begin{align*}
    \partial_t \| w \|^2_{p_{b}}
    &= \sum_{\ell \ge 1} \int_0^\infty
    \partial_\xi
    \Big(
    |w_\ell(t,\xi)|^2 + |\partial_\xi w_\ell(t,\xi)|^2
    \Big)
    |p_{A,b}(\xi)|^2 \dd \xi \\
    &\quad + K r_{\stat} \sum_{\ell \ge 1}
      \int_0^\infty
      \Re\Big[
      w_{\ell-1}(t,\xi)\conj{w_{\ell}(t,\xi)} - w_{\ell+1}(t,\xi) \conj{w_{\ell}(t,\xi)}
      \Big]
      |p_{A,b}(\xi)|^2 \dd \xi\\
    &\quad + K r_{\stat} \sum_{\ell \ge 1}
      \int_0^\infty
      \Re\Big[
      (\partial_\xi w_{\ell-1}(t,\xi)) \conj{(\partial_\xi w_{\ell}(t,\xi))}
      - (\partial_\xi w_{\ell+1}(t,\xi)) \conj{\partial_\xi w_{\ell}(t,\xi))}
      \Big]
      |p_{A,b}(\xi)|^2 \dd \xi.
  \end{align*}
  The first term is non-negative, because $p_b$ is non-decreasing. In
  the sums of the second and third term, the terms cancel after
  shifting one summand recalling the convention $w_0 \equiv 0$. Hence
  \begin{equation*}
    \partial_t \| w \|^2_{p_{b}} \le 0
  \end{equation*}
  and the same estimate holds for $L_1$.

  The result now follows from an approximation scheme and a standard
  compactness argument. We construct approximate solutions $w^n$ by
  restricting the evolution to $\ell \in [1,n]$ and smooth initial
  data with compact support $v^n$ with $v^n \to v$ in $\sX_{p_b}$. By
  the a priori estimate, these solutions satisfy $\| w^n \|_{p_b} \le
  \| v^n \|_{p_b}$ for $t \in [s,T]$. Hence $\{w^n : n \in \N\}$ is a
  bounded set in $L^\infty([s,T], \sX_{p_b})$. By the weak
  compactness, we extract a weak solution $w$. This shows the
  existence of a solution.

  For the weak continuity use that
  \begin{equation}
    \label{eq:bound-derivative-well-posedness}
    \partial_t w \in L^\infty([s,T], \sY_{p_b,0}),
  \end{equation}
  where $\sY_{p_b,0}$ is the Hilbert space defined by the norm
  \begin{equation*}
    \| v \|^2_{\sY_{p_{b},0}}
    = \sum_{\ell \ge 1} \int_0^\infty |v_{\ell}(\xi)|^2 |p_b(\xi)|^2
    \dd \xi.
  \end{equation*}
  The space $\sX_{p_b}$ is dense in $\sY_{p_b,0}$ so that
  \eqref{eq:bound-derivative-well-posedness} implies that $w$ is
  weakly continuous by standard functional analysis, see e.g.\ Theorem
  2.1 in \cite{strauss-1966-continuity-of-functions}.

  By the uniqueness, the complex linearity of \(E\) shows that the
  solution operator \(\sol Est\) is well-defined and complex linear.

  The decay estimate follows from refined energy estimates. For these,
  we will only present the a priori estimates, which can be justified
  in the same way.

  Let $w_\ell(t,\xi) = (\sol Est v)_l(\xi)$ and $\nu(t) = r_{\stat}$
  in the case of $E=L_1$ and $\nu(t) = r_{\stat} + \eta(t)$ in the
  case of $E = B_1$. Then with the weight $\phi = p_{A,b}$ and
  recalling the convention $v_0 \equiv 0$ we find
  \begin{align*}
    \frac{\dd}{\dd t}
    &\sum_{\ell \ge 1} \int_0^\infty
    |w_\ell(t,\xi)|^2 |\phi(\xi+t-s)|^2 \ell^{-1} \dd \xi \\
    &=
    \sum_{\ell \ge 1} \int_0^\infty
    \partial_\xi \left( |w_\ell(t,\xi)|^2 \right) |\phi(\xi+t-s)|^2
    \dd \xi
    +
    \sum_{\ell \ge 1} \int_0^\infty
    |w_\ell(t,\xi)|^2 \partial_\xi\left(|\phi(\xi+t-s)|^2\right)
      \ell^{-1} \dd \xi \\
    &\quad+ K \sum_{\ell \ge 1} \int_0^\infty
      \Re\left[
      \conj{\nu(t)} w_{\ell-1}(t,\xi) \conj{w_{\ell}(t,\xi)}
      - \nu(t) w_{\ell+1}(t,\xi) \conj{w_{\ell}(t,\xi)}
      \right] |\phi(\xi+t-s)|^2 \dd \xi \\
    &= - \sum_{\ell \ge 1} |w_\ell(t,0)|^2 |\phi(t-s)|^2
      - \sum_{\ell \ge 1} \int_0^\infty
      |w_\ell(t,\xi)|^2 \partial_\xi\left(|\phi(\xi+t-s)|^2\right)
      (1-\ell^{-1}) \dd \xi \\
    &\le - |w_1(t,0)|^2 |\phi(t-s)|^2,
  \end{align*}
  where we used that $\phi$ has non-negative derivative as $b \ge 0$
  and
  \begin{equation*}
    \sum_{\ell \ge 1}
    \Re\left[
      \conj{\nu(t)} w_{\ell-1}(t,\xi) \conj{w_{\ell}(t,\xi)}
    \right]
    =
    \sum_{\ell \ge 1}
    \Re\left[
      \nu(t) \conj{w_{\ell}(t,\xi)} w_{\ell+1}(t,\xi)
    \right].
  \end{equation*}
  Likewise we control $\partial_\xi w$ and as
  $\phi(\xi{+}t{-}s) = p_{A+t-s,b}(\xi)$ the claimed result
  follows.
\end{proof}

In preparation of Lemma \ref{thm:control-seminorms-by-weighted-sobolev}, we
first prove the following lemma.
\begin{lemma}
  \label{thm:beta-alpha-control-regularity-induction}
  Let $k = -1/2,-1,-3/2,\dots$ and $b > 1/2 + |k|$. Then there exists
  a constant $C_k$ only depending on $k$, $b$ and $K r_{\stat}$ such
  that
  \begin{equation*}
    \beta_{\alpha}(v)
    \le C_k A^{|k| - b} \| v \|_{p_{A,b},k}
  \end{equation*}
  and
  \begin{equation*}
    C_{-1/2} = \frac{1}{\sqrt{2b-1}}.
  \end{equation*}
\end{lemma}
\begin{proof}
  We prove it by induction over $k$ starting at $k=-1/2$ and going
  downwards. The base case is a simple application of
  Lemma \ref{thm:evolution-l1-b1} as by the Cauchy-Schwarz inequality
  \begin{align*}
    \left(\int_0^\infty \left| (\ee^{tL_1} v)_1(0) \right| \dd t\right)^2
    &\le
    \left(
      \int_0^\infty \left| (\ee^{tL_1} v)_1(0) \right|^2 (A+t)^{2b} \dd t
    \right)
    \left(
      \int_0^\infty (A+t)^{-2b} \dd t
      \right) \\
    &\le \frac{A^{1-2b}}{2b-1} \| v \|^2_{p_{A,b}}.
  \end{align*}

  For the induction step we use that the transport evolution is
  regularising in the spatial modes $\ell$ at the expense of a power
  in $\xi$. Assuming it is true for $k + 1/2$, we look at $k$ and find
  with the notation $v(t) = \ee^{tL_1} v$
  \begin{align*}
    &\frac{\dd}{\dd t} \| v(t) \|^2_{p_{A,b},k} \\
    &= \sum_{\ell \ge 1} \int_0^\infty
    \partial_\xi
    \Big(
    |v_\ell(t,\xi)|^2 + |\partial_\xi v_\ell(t,\xi)|^2
    \Big)
    |p_{A,b}(\xi)|^2 \ell^{2k+1} \dd \xi \\
    &\quad + K r_{\stat} \sum_{\ell \ge 1}
      \int_0^\infty
      \Re\Big[
      v_{\ell-1}(t,\xi)\conj{v_{\ell}(t,\xi)} - v_{\ell+1}(t,\xi) \conj{v_{\ell}(t,\xi)}
      \Big]
      |p_{A,b}(\xi)|^2 \ell^{2k+1} \dd \xi\\
    &\quad + K r_{\stat} \sum_{\ell \ge 1}
      \int_0^\infty
      \Re\Big[
      (\partial_\xi v_{\ell-1}(t,\xi)) \conj{(\partial_\xi v_{\ell}(t,\xi))}
      - (\partial_\xi v_{\ell+1}(t,\xi)) \conj{\partial_\xi v_{\ell}(t,\xi))}
      \Big]
      |p_{A,b}(\xi)|^2 \ell^{2k+1} \dd \xi \\
    &\le - A^{2b} |v_1(t,\xi)|^2
      - 2b \sum_{\ell \ge 1} \int_0^\infty
      \Big(
      |v_\ell(t,\xi)|^2 + |\partial_\xi v_\ell(t,\xi)|^2
      \Big)
      |p_{A,b-\frac 12}(\xi)|^2 \ell^{2k+1} \dd \xi \\
    &\quad +
      K r_{\stat} \sum_{\ell \ge 1} \int_0^\infty
      \Re\Big[
      v_{\ell}(t,\xi) \conj{v_{\ell+1}(t,\xi)}
      + (\partial_\xi v_{\ell}(t,\xi)) \conj{(\partial_{\xi} v_{\ell+1}(t,\xi))}
      \Big]
      ((\ell{+}1)^{2k+1} {-} \ell^{2k+1}) |p_{A,b}(\xi)|^2 \dd \xi.
  \end{align*}
  As $|(\ell{+}1)^{2k+1} - \ell^{2k+1}| \le |2k+1| \ell^{2k}$, this
  means that there exists a constant $C$ such that
  \begin{equation*}
    \frac{\dd}{\dd t}  \| v(t) \|^2_{p_{A,b},k}
    \le - A^{2b} |v_1(t,0)|^2
    - 2b \| v(t) \|_{p_{A,b-\frac 12},k+\frac 12}
    + C \| v(t) \|^2_{p_{A,b},k}.
  \end{equation*}
  Therefore,
  \begin{equation}
    \label{eq:beta-alpha-induction-lemma-regularisation}
    \| v \|^2_{p_{A,b},k}
    \ge \int_0^\infty \ee^{-Ct}
    \left[
      A^{2b} |v_1(t,0)|^2
      + 2b \| v(t) \|_{p_{A,b-\frac 12},k+\frac 12}
    \right]
    \dd t.
  \end{equation}
  Hence there exists a time $t^* \in [0,1]$ such that
  \begin{equation*}
    \| \ee^{t^*L_1} v \|^2_{p_{A,b-\frac 12},k+\frac 12}
    \le \frac{\ee^{C}}{2b}
    \| v \|^2_{p_{A,b},k}.
  \end{equation*}
  Using the semigroup property, we find with $v^* = \ee^{t^*L_1} v$
  \begin{equation*}
    \beta_{\alpha}(v)
    = \int_0^{t^*} |v_1(t,0)|\, \dd t
    + \int_0^\infty |(\ee^{tL_1} v^*)_1(0)|\, \dd t.
  \end{equation*}
  The first term can be controlled by
  \eqref{eq:beta-alpha-induction-lemma-regularisation} and the second
  term by the induction hypothesis. This proves the induction step.
\end{proof}

The bounds on the seminorms are now an easy consequence.
\begin{proof}[Proof of Lemma \ref{thm:control-seminorms-by-weighted-sobolev}]
  The bound on $\beta_{\alpha}$ is already proved in the previous
  lemma. For $\beta_{d}$ we use the previous lemma with $k=-3/2$ as
  \begin{align*}
    \| (\ell u_{\ell-1})_{\ell} \|_{p_b} &\le \| u \|_{p_b,-\frac 32}, \\
    \| (\ell u_{\ell})_{\ell} \|_{p_b} &\le \| u \|_{p_b,-\frac 32}, \\
    \| (\ell u_{\ell+1})_{\ell} \|_{p_b} &\le  2\| u \|_{p_b,-\frac 32}.
  \end{align*}

  Finally, the control of $\beta_{\eta}$ is a consequence of the
  Sobolev embedding theorem, which implies a constant $C_S$ such that
  a function $u_1$ of $\xi$ satisfies
  \begin{equation*}
    u_1(0) \le C_S \int_0^1
    \left(
      |u_1(\xi)|^2 + |\partial_\xi u_1(\xi)|^2
    \right)
    \dd \xi. \qedhere
  \end{equation*}
\end{proof}

\section{Volterra equation for linearised evolution}
\label{sec:linear-volterra}

For controlling the order parameter, we use the theory of the Volterra
equation and follow the setup of the book by
\textcite{gripenberg-londen-staffans-1990-volterra}.

For the convolution kernel \(k_{Lc}\) of the linearised evolution, let
\begin{equation*}
  k_{Lc,r}(\tau) = -(\ee^{(t-s)L_1} r_r)_1(0) \quad\text{and}\quad
  k_{Lc,i}(\tau) = -(\ee^{(t-s)L_1} r_i)_1(0)
\end{equation*}
so that the convolution kernel \(k_{Lc}\) can be written as
\begin{equation*} \setlength\arraycolsep{4pt}
  k_{Lc}(\tau) =
  \begin{pmatrix}
    \Re k_{Lc,r}(\tau) & \Re k_{Lc,i}(\tau) \\
    \Im k_{Lc,r}(\tau) & \Im k_{Lc,i}(\tau)
  \end{pmatrix}.
\end{equation*}

Using the decay under $L_1$, we find the bound on the kernel.
\begin{proof}[Proof of Lemma \ref{thm:decay-linear-volterra-kernel}]
  By the Cauchy Schwarz inequality it holds that
  \begin{equation*}
    \left(
      \int_0^\infty |k_{Lc,r}(t)| (1+t)^b\, \dd t
    \right)^2
    \le
    \left(
      \int_0^\infty |k_{Lc,r}(t)|^2 (1+t)^{2b_r}\, \dd t
    \right)
    \left(
      \int_0^\infty (1+t)^{2(b-b_r)}\, \dd t
    \right)
  \end{equation*}
  and likewise for $k_{Lc,i}$.  As $b_r > b + 1/2$, the estimate of
  Lemma~\ref{thm:evolution-l1-b1} implies the result.
\end{proof}

The rotation symmetry implies that $\DD \hat{R} \hat{f}_\stat$ is a
zero eigenmode of the linearised evolution and its corresponding order
parameter \((\DD \hat{R} \hat{f}_\stat)_1(0)\) is $\ii r_{\stat}$,
i.e.\ purely imaginary.

Next we prove that the linear stability from Definition
\ref{def:linear-stability} is sharp for the original system without
the complexification, similar to the results in
\cite{dietert-fernandez-gerard-varet-2018-landau-damping-pls}. We
prove the results directly, because we need to handle the case of
poles at $\Re z =0$, which can be the boundary of the resolvent set.

\begin{proposition}
  Let $b_r \ge 0$ and assume the stationary state $\hat{f}_{\stat}$ is
  regular enough such that $\| r_{r} \|_{p_{b_r}} < \infty$ and
  $\| r_{i} \|_{p_{b_r}} < \infty$.

  If $\lambda$ is a root of the characteristic equation with
  $\Re \lambda > 0$ and $\Im \lambda = 0$, then there exists an
  eigenmode $v_\lambda$ with $\| v_{\lambda} \|_{p_{b_r}} < \infty$,
  i.e.\ satisfying $Lv_{\lambda} = \lambda v_{\lambda}$.

  If $\lambda$ is a root of the characteristic equation with
  $\Re \lambda > 0$ and $\Im \lambda \not= 0$, then also
  $\conj{\lambda}$ is a root of the characteristic equation and there
  exist modes $v_{\lambda,c}$ and $v_{\lambda,s}$ with
  $\| v_{\lambda,c} \|_{p_{b_r}} < \infty$ and
  $\| v_{\lambda,c} \|_{p_{b_r}} < \infty$ satisfying
  \begin{align*}
    L v_{\lambda,c} &= (\Re \lambda) v_{\lambda,c} - (\Im \lambda) v_{\lambda,s}, \\
    L v_{\lambda,s} &= (\Im \lambda) v_{\lambda,c} + (\Re \lambda) v_{\lambda,s}.
  \end{align*}

  If $\lambda$ is a root of the characteristic equation with
  $\Re \lambda = 0$ and $b_r > 1$, then the above modes exist with the
  bound $\| v_{\lambda} \|_{p_b} < \infty$ and
  $\| v_{\lambda,c} \|_{p_b} < \infty$,
  $\| v_{\lambda,s} \|_{p_b} < \infty$ for $0 \le b < b_r -1$.

  If $b_r > 2$ and $\lambda=0$ is not a simple root, i.e.
  \begin{equation*}
    \frac{\dd}{\dd z}
    \det\Big(\Id + (\Lap k_{Lc})(z)\Big)
    \Bigg|_{z=\lambda}
    = 0
  \end{equation*}
  then at least one of the following possibilities holds:
  \begin{itemize}
  \item There exist two eigenmodes $v_{0,r}$ and $v_{0,i}$
    with $\| v_{0,r} \|_{p_b} < \infty$ and
    $\| v_{0,r} \|_{p_b} < \infty$ for $0 \le b < b_r -1$.
  \item There exist two modes $v_{0,0}$ and $v_{0,1}$ with
    $\| v_{0,0} \|_{p_{p}} < \infty$ and
    $\| v_{0,1} \|_{p_{b-1}} < \infty$ for $1 \le b < b_r-1$
    satisfying
    \begin{equation*}
      L v_{0,0} = 0
      \text{ and }
      L v_{0,1} = v_{0,0}.
    \end{equation*}
  \end{itemize}
\end{proposition}
\begin{proof}
  If $\lambda$ is satisfying $\Im \lambda = 0$ and
  $\Re \lambda \ge 0$, then $\Lap k_{Lc}$ is a real matrix. Hence if
  $\lambda$ is a root, there exists $w_r,w_i\in\R$ such that
  \begin{equation*}
    \begin{pmatrix}
      w_r \\ w_i
    \end{pmatrix}
    \in \ker [1 + \Lap k_{Lc}(\lambda)].
  \end{equation*}
  Then define the mode $v_{\lambda}$ by
  \begin{equation*}
    v_{\lambda} = \int_0^\infty \ee^{tL_1} (w_r r_r + w_i r_i)
    \ee^{-\lambda t} \dd t,
  \end{equation*}
  which is a converging Bochner integral with the claimed bounds by
  Lemma~\ref{thm:evolution-l1-b1}.

  Moreover, we find
  \begin{equation*}
    \begin{pmatrix}
      \Re (v_{\lambda})_0(1) \\
      \Im (v_{\lambda})_0(1)
    \end{pmatrix}
    = - \Lap k_{Lc}(\lambda)
    \begin{pmatrix}
      w_r \\ w_i
    \end{pmatrix}
    =
    \begin{pmatrix}
      w_r \\ w_i
    \end{pmatrix},
  \end{equation*}
  so that
  \begin{equation*}
    L_2 v_{\lambda} = w_r r_r + w_i r_i.
  \end{equation*}
  On the other hand
  \begin{equation*}
    L_1 v_{\lambda} = \lambda v_{\lambda} - (w_r r_r + w_i r_i),
  \end{equation*}
  which shows that $v_{\lambda}$ is the claimed eigenmode.
  In the case of a root $\lambda$ with $\Re \lambda \ge 0$ and $\Im
  \lambda \not = 0$, let
  \begin{equation*}
    \begin{pmatrix}
      w_1 \\ w_2
    \end{pmatrix}
    \in \ker(1+\Lap k_{Lc}(\lambda)).
  \end{equation*}
  Taking the conjugate shows that
  \begin{equation*}
    \begin{pmatrix}
      \conj{w_1} \\ \conj{w_2}
    \end{pmatrix}
    \in \ker(1+\Lap k_{Lc}(\conj{\lambda})),
  \end{equation*}
  so that $\conj{\lambda}$ is also a root. Then define the modes as
  \begin{align*}
    v_{\lambda,c} &= \int_0^\infty \ee^{tL_1}
    \Big[ (w_1\ee^{-\lambda t} + \conj{w_1} \ee^{-\conj{\lambda} t}) r_r
    + (w_2\ee^{-\lambda t} + \conj{w_2} \ee^{-\conj{\lambda} t}) r_i
    \Big]
    \dd t, \\
    v_{\lambda,s} &= \int_0^\infty \ee^{tL_1}
    \Big[ -\ii (w_1\ee^{-\lambda t} + \conj{w_1} \ee^{-\conj{\lambda} t}) r_r
    - \ii (w_2\ee^{-\lambda t} + \conj{w_2} \ee^{-\conj{\lambda} t}) r_i
    \Big]
    \dd t,
  \end{align*}
  which satisfy the claimed bounds. Moreover, as before
  \begin{equation*}
    \begin{pmatrix}
      \Re (v_{\lambda,c})_1(0) \\
      \Im (v_{\lambda,c})_1(0)
    \end{pmatrix}
    =
    \begin{pmatrix}
      w_1 + \conj{w_1}\\
      w_2 + \conj{w_2}
    \end{pmatrix}
    \quad\text{ and }\quad
    \begin{pmatrix}
      \Re (v_{\lambda,s})_1(0) \\
      \Im (v_{\lambda,s})_1(0)
    \end{pmatrix}
    =
    \begin{pmatrix}
      -\ii(w_1 - \conj{w_1})\\
      -\ii(w_2 - \conj{w_2})
    \end{pmatrix}.
  \end{equation*}
  Therefore, we find directly
  \begin{align*}
    L v_{\lambda,c} &= (\Re \lambda) v_{\lambda,c} - (\Im \lambda) v_{\lambda,s}, \\
    L v_{\lambda,s} &= (\Im \lambda) v_{\lambda,c} + (\Re \lambda) v_{\lambda,s},
  \end{align*}
  which is the claimed relation.

  For the case \(\Re \lambda = 0\) notice the bound
  \begin{equation*}
    \| \ee^{t L_1} r_r \|_{p_b}
    \le (1+t)^{b-b_r} \| \ee^{t L_1} r_r \|_{p_{1+t,b_r}}
  \end{equation*}
  directly from the definition of the weighted norm (likewise for
  \(r_i\)). Hence the result follows again with the claimed bounds as
  we take \(b<b_r-1\).

  In the case that $\lambda$ is not a simple root with
  $\Re \lambda \ge 0$ and $\Im \lambda = 0$, one possibility is that
  $1 + \Lap k_{Lc}(\lambda) = 0$. In this case, we have the following
  two eigenmodes
  \begin{equation*}
    v_{\lambda,r} = \int_0^\infty \ee^{tL_1} r_r \ee^{-\lambda t} \dd t
    \qquad\text{ and }\qquad
    v_{\lambda,i} = \int_0^\infty \ee^{tL_1} r_i \ee^{-\lambda t} \dd t.
  \end{equation*}
  Otherwise, find $a_r$ and $a_i$ such that
  \begin{equation*}
    \Big[\adjoint (\Id+\Lap k_{Lc}(\lambda))\Big]
    \begin{pmatrix}
      a_r \\ a_i
    \end{pmatrix}
    =:
    \begin{pmatrix}
      w_r \\ w_i
    \end{pmatrix}
    \not = 0
  \end{equation*}
  with the adjoint
  \begin{equation*}
    \adjoint (\Id+\Lap k_{Lc}(\lambda))
    =
    \begin{pmatrix}
      1 + \Lap \Im k_{Lc,i}(\lambda)
      & - \Lap\Re k_{Lc,i}(\lambda) \\
      -\Lap \Im k_{Lc,r}(\lambda)
      & 1 + \Lap \Re k_{Lc,r}(\lambda)
    \end{pmatrix}.
  \end{equation*}
  Let
  \begin{equation*}
    \begin{pmatrix}
      w_r' \\ w_i'
    \end{pmatrix}
    = \left.\frac{\dd}{\dd z}
      \adjoint (\Id+\Lap k_{Lc}(z))\right|_{z=\lambda}
    \begin{pmatrix}
      a_r \\ a_i
    \end{pmatrix}
  \end{equation*}
  and define the modes
  \begin{align*}
    v_{\lambda,0}
    &= \int_0^\infty \ee^{tL_1} (w_r r_r + w_i r_i) \ee^{-\lambda t}
      \dd t \\
    v_{\lambda,1}
    &= \int_0^\infty \ee^{tL_1}
      \Big[(-t w_r + w_r') r_r + (-t w_i + w_i') r_i\Big]
       \ee^{-\lambda t} \dd t.
  \end{align*}
  As before, we have that $L v_{\lambda,0} = \lambda v_{\lambda,0}$
  and for the mode $v_{\lambda,1}$ we find that
  \begin{equation*}
    \begin{pmatrix}
      \Re (v_{\lambda,1})_1(0) \\
      \Im (v_{\lambda,1})_1(0)
    \end{pmatrix}
    = - \left.\frac{\dd}{\dd z} (\Id{+}\Lap k_{Lc}(z))\right|_{z=\lambda}
    \begin{pmatrix}
      w_r \\ w_i
    \end{pmatrix}
    - \Lap k_{Lc}(\lambda)
    \begin{pmatrix}
      w'_r \\ w'_i
    \end{pmatrix}.
  \end{equation*}
  With $M(z) = \Id + \Lap k_{Lc}(z)$, this can be written as
  \begin{align*}
    \begin{pmatrix}
      \Re (v_{\lambda,1})_1(0) \\
      \Im (v_{\lambda,1})_1(0)
    \end{pmatrix}
    &= -
    \left\{
      \left.\frac{\dd}{\dd z} M(z) \right|_{z=\lambda}
      \adjoint M(\lambda)
      + M(\lambda)
      \left.\frac{\dd}{\dd z} \adjoint M(z)\right|_{z=\lambda}
      \right\}
      \begin{pmatrix}
        a_r \\ a_i
      \end{pmatrix}
    +
    \begin{pmatrix}
      w_r' \\ w_i'
    \end{pmatrix} \\
    &= -
    \left.\frac{\dd}{\dd z} \det M(z) \right|_{z=\lambda}
    \begin{pmatrix}
      a_r \\ a_i
    \end{pmatrix}
    +
    \begin{pmatrix}
      w_r' \\ w_i'
    \end{pmatrix}.
  \end{align*}
  As we assume that $\lambda$ is not a simple root of $\det(\Id+\Lap k_{Lc})$,
  the first term vanishes and we find that
  \begin{equation*}
    \begin{pmatrix}
      \Re (v_{\lambda,1})_1(0) \\
      \Im (v_{\lambda,1})_1(0)
    \end{pmatrix}
    =
    \begin{pmatrix}
      w_r' \\ w_i'
    \end{pmatrix}.
  \end{equation*}
  Then we can directly verify as before that
  \begin{equation*}
    L v_{\lambda,1} = \lambda v_{\lambda,1} + v_{\lambda,0},
  \end{equation*}
  which is the claimed relation.
\end{proof}

The weights $p_{A,b}$ satisfy a refined submultiplicative as
\begin{equation*}
  p_{A,b}(s+t) \le p_{1,b}(s)\, p_{A,b}(t)
  \qquad \text{for $s,t \in \R^+$},
\end{equation*}
for $A \ge 1$ and $b \ge 0$. This allows to control the
convolution with a refined Young inequality.
\begin{lemma}
  \label{thm:young-inequality}
  Let $\alpha \in L^1(\R^+,p_{b})$ and
  $\beta \in L^\infty(\R^+,p_{A,b})$ with $A\ge 1$ and $b \ge 1$, then
  \begin{equation*}
    \| \alpha \conv \beta \|_{L^\infty(\R^+,p_{A,b})}
    \le \| \alpha \|_{L^1(\R^+,p_b)}\,
    \| \beta \|_{L^\infty(\R^+,p_{A,b})}.
  \end{equation*}
\end{lemma}
\begin{proof}
  The result follows directly from the refined submultiplicativity,
  see \cite[Lemma~19]{dietert-2016-stability-bifurcation}.
\end{proof}

If the kernel is sufficiently decaying, then the single root of the
characteristic equation at $z=0$ must behave like a pole and can be
separated.
\begin{proposition}
  \label{thm:linear-resolvent-bound}
  Let $b \ge 0$ and $b_r > b + 5/2$. Assume that $\hat{f}_{\stat}$ is
  such that
  \begin{equation*}
    \| r_{r} \|_{p_{b_r}} < \infty
    \qquad\text{ and }\qquad
    \| r_{i} \|_{p_{b_r}} < \infty
  \end{equation*}
  and that $k_{Lc}$ satisfies the stability condition from
  Definition~\ref{def:linear-stability}. Then the resolvent $r_{Lc}$
  takes the form
  \begin{equation*}
    r_{Lc} = K_{\Theta} + r_{Lcs},
  \end{equation*}
  where $\| r_{Lcs} \|_{L^1(\R^+,p_b)} < \infty$ and $K_{\Theta}$ is a
  constant matrix.
\end{proposition}
\begin{proof}
  We use Section 3 of Chapter 7 of
  \cite{gripenberg-londen-staffans-1990-volterra}, which applies as
  our weight is submultiplicative.

  By the assumed regularity on $\hat{f}_{\stat}$, we have that
  \begin{equation*}
    \int_0^\infty \| k_{Lc}(t) \|\, (1+t)^{b+2}\, \dd t < \infty
  \end{equation*}
  so that $k_{Lc}$ is smooth of order at least 2 in $\hat{L}_1(p_b)$
  with the Definitions 3.1 and 3.5 of Chapter 7 of
  \cite{gripenberg-londen-staffans-1990-volterra}, see also Lemma 4.3
  of \cite{jordan-staffans-wheeler-1982-local-analyticity}. By the
  condition on the derivative, $1+\Lap k_{Lc}$ has a zero of order 1
  (see Definition~3.6 of Chapter~6 of
  \cite{gripenberg-londen-staffans-1990-volterra}). Hence by the
  corresponding version of Theorem~3.7 of Chapter~7 of
  \cite{gripenberg-londen-staffans-1990-volterra} or Theorem 3.6 of
  \cite{jordan-staffans-wheeler-1982-local-analyticity} the result
  follows.
\end{proof}

We can identify $K_\Theta$ and the corresponding functinal \(\alpha\)
from \eqref{eq:stable-subspace-functional-alpha} more precisely:
\begin{lemma}
  \label{thm:resolvent-eigenmode-form}
  Assume the setup of Proposition \ref{thm:linear-resolvent-bound} and
  \begin{equation*}
    \| r_{\Theta} \|_{p_{b_r}} < \infty.
  \end{equation*}
  Then $K_{\Theta}$ can be written as
  \begin{equation*}
    K_{\Theta}
    =
    \begin{pmatrix}
      0 \\ 1
    \end{pmatrix}
    \otimes
    \begin{pmatrix}
      c_{\Theta,r} & c_{\Theta,i}
    \end{pmatrix}
  \end{equation*}
  for constants $c_{\Theta,r}$ and $c_{\Theta,i}$. Moreover, its
  kernel is determined by
  \begin{equation*}
    K_{\Theta} \; (\Id+\Lap k_{Lc}(0)) = 0
  \end{equation*}
  and
  \begin{equation*}
    \alpha(\DD \hat{R} \hat{f}_{\stat}) \not = 0.
  \end{equation*}
\end{lemma}
\begin{proof}
  Consider the rotation eigenmode and its forcing
  \begin{equation*}
    F_{L,\Theta}(t) = (\ee^{t L_1} \DD \hat{R} \hat{f}_{\stat})_1(0),
  \end{equation*}
  which is decaying as $(1+t)^{-b_r}$ by Lemma~\ref{thm:evolution-l1-b1}
  and Lemma~\ref{thm:control-seminorms-by-weighted-sobolev}, because
  $\DD \hat{R} f_{\stat}$ is proportional to $r_{\Theta}$. As it is an
  eigenmode, the order parameter is the constant
  $(\DD \hat{R} f_{\stat})_1(0) = \ii r_{\stat}$, so that
  \begin{equation*}
    \begin{pmatrix}
      0 \\ 1
    \end{pmatrix}
    = \compF_{L,\Theta} - r_{Lc} \conv F_{L,\Theta}
    = \compF_{L,\Theta} - K_{\Theta} \conv F_{L,\Theta} - r_{Lcs}
    \conv F_{L,\Theta}.
  \end{equation*}
  By Lemma \ref{thm:young-inequality}, the term
  $r_{Lcs} \conv F_{L,\Theta}$ is also vanishing as $t\to
  \infty$. Therefore we find from the limit $t\to \infty$ that
  \begin{equation}
    \label{eq:range-k-theta-rotation}
    \begin{pmatrix}
      0 \\ 1
    \end{pmatrix}
    = - \int_0^\infty K_{\Theta} \compF_{L,\Theta}(t) \dd t.
  \end{equation}
  This shows that
  \begin{equation*}
    \begin{pmatrix}
      0 \\ 1
    \end{pmatrix}
    \in \range K_{\Theta}.
  \end{equation*}

  Taking the Laplace transform of
  \eqref{eq:linear-resolvent-definition} shows that
  \begin{equation*}
    (\Id+\Lap k_{Lc}(0))\; K_{\Theta} = 0
    \qquad\text{ and }\qquad
    K_{\Theta}\; (\Id+\Lap k_{Lc}(0)) = 0.
  \end{equation*}
  The stability condition implies that $\Id+\Lap k_{Lc}(0) \not = 0$
  because otherwise $\det (\Id+\Lap k_{Lc}(z))$ would have a root of
  order at least two. Therefore, the range of $K_{\Theta}$ must be
  one-dimensional and $K_{\Theta}$ takes the given form. Moreover,
  \eqref{eq:range-k-theta-rotation} then implies that
  $\alpha(\DD \hat{R} \hat{f}_{\stat}) \not = 0$.

\end{proof}

In this setting, we find for a forcing $\compF_{L}$ from the
complexification of by $F_{L}$ through
\begin{equation*}
  \compF_{L} =
  \begin{pmatrix}
    \Re F_L \\ \Im F_L
  \end{pmatrix}
\end{equation*}
that
\begin{equation*}
  K_{\Theta} \conv \compF_{L}(t)
  =
  \begin{pmatrix}
    0 \\ 1
  \end{pmatrix}
  \int_0^t
  \Big(c_{\Theta,r} \Re F_{L}(s) + c_{\Theta,i} \Im F_{L}(s)\Big)\, \dd s.
\end{equation*}
By the solution formula \eqref{eq:linear-volterra-resolvent-solution},
the order parameter can therefore only decay if
\begin{equation*}
  \int_0^\infty
  \Big(c_{\Theta,r} \Re F_{L}(t) + c_{\Theta,i} \Im F_{L}(t)\Big)\, \dd t
  = 0,
\end{equation*}
which motivates the definition of $\alpha$ in
\eqref{eq:stable-subspace-functional-alpha}. Precisely, we find:
\begin{lemma}
  \label{thm:alpha-control-integral-imaginary}
  Let $E = L_1$ and $\| u_{\init} \|_{p_b} < \infty$ for $b > 1/2$ or
  let $E = B$ and $\| u_{\init} \|_{p_b} < \infty$ and
  $\| r_{\Theta} \|_{p_b} < \infty$ for $b > 3/2$. Then for
  $t \in \R^+$ it holds that
  \begin{equation*}
    \alpha(\sol{E}0t u_{\init}) + \int_0^t
    \Big(c_{\Theta,r} \Re (\sol E0s u_{\init})_1(0)
    + c_{\Theta,i} \Im (\sol E0s u_{\init})_1(0)\Big)
    \, \dd s
    = \alpha(u_{\init}).
  \end{equation*}
\end{lemma}
\begin{proof}
  Note that
  \begin{equation*}
    \frac{\dd}{\dd t} \alpha( \sol{E}0t u_{\init} )
    = \alpha[L_1 (\sol{E}0t u_{\init})],
  \end{equation*}
  from where the result follows.
\end{proof}
For this kind of forcing we can therefore formulate the following
corollary, where we control \(\alpha\) by the seminorm
\(\beta_{\alpha}\) (recall \eqref{eq:def-seminorm-ba} for the
definition).
\begin{corollary}
  \label{thm:forcing-control-eigenmode-b-alpha}
  Let $E = L_1$ and $\| u_{\init} \|_{p_b} < \infty$ for $b > 1/2$ or
  let $E = B$ and $\| u_{\init} \|_{p_b} < \infty$ for $b >
  3/2$. Furthermore, assume the setup of
  Lemma~\ref{thm:resolvent-eigenmode-form}. If $\alpha(u_{\init}) = 0$,
  then the forcing
  \begin{equation*}
    \compF(t) =
    \begin{pmatrix}
      \Re F(t) \\
      \Im F(t)
    \end{pmatrix}
    \qquad\text{with}\qquad
    F(t) = (\sol E0t u_\init)_1(0)
  \end{equation*}
  satisfies for $t \ge 0$ that
  \begin{equation*}
    |(K_{\Theta} \conv F)(t)|
    \le \| K_{\Theta} \| \, \beta_{\alpha}(\sol E0t u_{\init}).
  \end{equation*}
\end{corollary}
\begin{proof}
  By the previous lemma, we can estimate
  \begin{equation*}
    \begin{split}
      \left|
        \int_0^t
        \Big(c_{\Theta,r} \Re (\sol E0s u_{\init})_1(0)
        + c_{\Theta,i} \Im (\sol E0s u_{\init})_1(0)\Big)
        \, \dd s
      \right|
      &= |\alpha(\sol E0t u_{\init})| \\
      &\le \| K_{\Theta} \| \,\beta_{\alpha}(\sol E0t u_{\init}),
    \end{split}
  \end{equation*}
  where the control by \(\beta_{\alpha}\) follows directly from the
  definition of \(\alpha\) and \(\beta_{\alpha}\).
\end{proof}
The contribution of the stable part $r_{Lcs}$ can easily be controlled
by Lemma~\ref{thm:young-inequality}.

\section{Nonlinear forcing and Volterra kernel}
\label{sec:nonlinear-forcing}

In this section we study the evolution under the time-dependent
operator $B$ and the difference to the evolution $L_1$ of the
linearised dynamics in order to estimate the deviation in the Volterra
kernel.

By the seminorm \(\beta_d\), we can bound the effect of \(B_2\) and
introduce \(B\).

\begin{proof}[Proof of Lemma \ref{thm:time-evolution-b}]
  Recall that \(B_2(w) = - \alpha(B_{1n} w)\, r_{\Theta}\) so that by
  the Definition~\eqref{eq:def-b1n} of \(B_{1n}\) the bootstrap
  assumption~\eqref{eq:bootstrap-def-d} implies
  \begin{equation*}
    |\alpha(w)|
    \le M_d(t)\; \| K_{\Theta} \|\; \beta_d(w).
  \end{equation*}

  The bound of $\beta_d$ from
  Lemma~\ref{thm:control-seminorms-by-weighted-sobolev} then implies
  that the operator $B_2$ is a bounded operator. Hence the evolution has a
  unique solution given by Duhamel's principle as
  \begin{equation}
    \label{eq:duhamel-b1-b}
    \sol Bst v  = \sol{B_1}st v + \int_s^t \sol{B_1}{\tau}{t} B_2(\sol
    Bs\tau v)\, \dd \tau.
  \end{equation}
  The real linearity follows directly from the real linearity of
  \(B_1\) and \(B_2\).

  The first estimates follow directly from this Duhamel representation
  using Lemma~\ref{thm:evolution-l1-b1}.

  Finally, under the assumption $b_r > b+1$ or $b_d > 1$ we find that
  \begin{equation*}
    \begin{split}
      &(1{+}t{-}s)^{\frac 32 - b} \int_s^t (1{+}t{-}\tau)^{\frac 32 - b_r}
      (1{+}\tau)^{-b_d} (1{+}\tau{-}s)^{\frac 32 - b}\, \dd \tau \\
      &\le
      \int_s^t (1{+}t{-}\tau)^{b - b_r} (1{+}\tau)^{-b_d}\, \dd \tau
    \end{split}
  \end{equation*}
  is uniformly bounded by a constant. Moreover, $\beta_d$ varies
  continuously by the weak-continuity. Hence by a bootstrap argument,
  we can find $\delta_R$ such that the claimed control holds.
\end{proof}

To control the behaviour under the time-dependent evolution \(B\), we
start with a simple lemma controlling the norm.
\begin{lemma}
  \label{thm:control-norm-evolution}
  Let $b > 3/2$, $b_d \ge 0$ and $b_r \ge b$ with
  $\| r_{\Theta} \|_{p_{b_r}} < \infty$. Assume that $b_r > b+1$ or
  $b_d > 1$, then there exist constants $\delta_R$ and $C$ such that
  \begin{equation*}
    \| \sol Est v \|_{p_{1+t-s,b}}
    \le C (1+t-s)^{\bar{b}_d} \| v \|_{p_{b}}
    \quad\text{ with }\quad
    \bar{b}_d = \max\left\{0,\frac 32 - b_d\right\}
  \end{equation*}
  for $E = L_1$ or $E = B$ if $R_d(t) \le \delta_R$.
\end{lemma}
\begin{proof}
  The case $E=L_1$ follows directly from Lemma~\ref{thm:evolution-l1-b1}.

  In the case $E=B$, we can chose $\delta_R$ small enough to apply
  Lemma \ref{thm:time-evolution-b} in order to control
  $\beta_d(\sol Est v)$.  By Duhamel's
  principle~\eqref{eq:duhamel-b1-b}, Lemma~\ref{thm:evolution-l1-b1}
  implies
  \begin{equation*}
    \begin{split}
      &\| \sol Bst v \|_{p_{1+t-s,b}} \\
      &\le \| v \|_{p_b}
      + 2 C_{\beta d} R_d(t) \int_s^t (1+\tau)^{-b_d} (1+\tau-s)^{\frac
        32 -b}\, \| v \|_{p_b} \| \sol B\tau t r_{\Theta}
      \|_{p_{1+t-s,b}}\, \dd \tau.
    \end{split}
  \end{equation*}
  By using Lemma~\ref{thm:evolution-l1-b1} again, we bound
  \begin{equation*}
    \begin{split}
      \| \sol B\tau t r_{\Theta} \|_{p_{1+t-s,b}}
      &\le (1+t-\tau)^{b-b_r} (1+\tau-s)^{b}\,
      \| \sol B\tau t r_{\Theta} \|_{p_{1+t-\tau,b_r}} \\
      &\le (1+t-\tau)^{b-b_r} (1+\tau-s)^{b}\, \| r_{\Theta} \|_{p_{b_r}}.
    \end{split}
  \end{equation*}
  Plugging in this bound gives the claimed result.
\end{proof}

We need to control the difference of the kernel elements under the
transport \(L_1\) and \(B\). As the difference is an unbounded
operator loosing regularity, we want to use that \(r_{r}\) and \(r_i\)
have more regularity.  Using the transport part, we propagate the
higher regularity.
\begin{lemma}
  \label{thm:higher-regularity-control}
  Let $b_0 \ge 0$, $b_d \ge 0$, $b \ge b_0+1/2$ with $b > 3/2$ and
  $b_r \ge b$ with
  \begin{equation*}
    \| r_{\Theta} \|_{p_{b_r}} < \infty
    \qquad\text{and}\qquad
    \| r_{\Theta} \|_{p_{b_r-\frac 12},0} < \infty.
  \end{equation*}
  Assume that one of the following conditions holds
  \begin{itemize}
  \item $b_d > 5/2$; or
  \item $b_r > b+1$ and
    \begin{equation*}
      b > b_0 + \frac 74 + \frac{\bar{b}_d - b_d}{2}
    \end{equation*}
    with $\bar{b}_d$ from Lemma \ref{thm:control-norm-evolution}.
  \end{itemize}
  Then there exist constants $\delta_R$ and $C$ such that
  \begin{equation*}
    \| \sol Est v \|_{p_{1+t-s,b_0},0}
    \le C ( \| v \|_{p_{b_0},0} + \| v \|_{p_b} )
  \end{equation*}
  for $E = L_1$ and $E = B$ if $R_d(t) \le \delta_R$.
\end{lemma}
\begin{proof}
  We present the proof for $E = B$ which is the more difficult
  part. The proof for $E = L_1$ follows from dropping the additional
  terms.

  We start with showing a $L^2$ in time control. For this note that
  \begin{equation*}
    \begin{split}
      &\frac{\dd}{\dd t}
      \| \sol Bst v \|^2_{p_{1+\frac{t-s}{2},b_0+\frac 12}} \\
      &\le - \left(b{+}\frac 12\right)
      \| \sol Bst v \|^2_{p_{1+\frac{t-s}{2},b_0},0} \\
      &\quad+ 2 \| K_\Theta\|\, R_d(t)\, (1{+}t)^{-b_d} \beta_d(\sol Bst v)\,
      \| r_{\Theta} \|_{p_{1+\frac{t-s}{2},b_0+\frac 12}}
      \| \sol Bst v \|_{p_{1+\frac{t-s}{2},b_0+\frac 12}}.
    \end{split}
  \end{equation*}
  By choosing $\delta_R$ small enough to apply Lemma \ref{thm:control-norm-evolution}, we
  control the growth term as
  \begin{align*}
    &2\| K_\Theta\|\,R_d(t)\, (1{+}t)^{-b_d} \beta_d(\sol Bst v)
    \| r_{\Theta} \|_{p_{1+\frac{t-s}{2},b_0+\frac 12}}
    \| \sol Bst v \|_{p_{1+\frac{t-s}{2},b_0+\frac 12}}\\
    &\le C \| K_\Theta\|\,R_d(t)\, (1{+}t)^{-b_d} (1{+}t{-}s)^{\frac 32-b}
    (1{+}t{-}s)^{b_0+\frac 12} (1{+}t{-}s)^{b_0+\frac 12 -b + \bar{b}_d}
    \| v \|^2_{p_b}
  \end{align*}
  for a constant $C$ with $\bar{b}_d$ from
  Lemma \ref{thm:control-norm-evolution} if $R_d(t) \le \delta_R$. By the
  assumptions, the term is integrable, so that there exists a constant
  $C$ such that
  \begin{align*}
    \int_s^t \| \sol Bs\tau v \|^2_{p_{1+t-s,b_0},0}\, \dd \tau
    \le 2^{2b_0+1} \int_s^t \| \sol Bs\tau v
    \|^2_{p_{1+\frac{t-\tau}{2},b_0},0}\, \dd \tau
    \le C \| v \|^2_{p_b}
  \end{align*}
  if $R_d(t) \le \delta_R$.

  By adapting the estimate from the proof of
  Lemma~\ref{thm:beta-alpha-control-regularity-induction}, there
  exists a constant $C$ such that
  \begin{align*}
    \frac{\dd}{\dd t} \| \sol Bst v \|^2_{p_{1+t-s,b_0},0}
    &\le C\, (r_{\stat} + |\eta(t)|)\,
    \| \sol Bst v \|^2_{p_{1+t-s,b_0},0}\\
    &+ \| K_{\Theta} \| R_d(t)\, (1+t)^{-b_d} \beta_d(\sol Bst v)\,
    \| r_{\Theta} \|_{p_{1+t-s,b_0},0}\, \| \sol Bst v \|_{p_{1+t-s,b_0},0}.
  \end{align*}
  In the second term $\beta_d(\sol Bst v)$ is controlled by
  $2C_{\beta d} (1+t-s)^{\frac 32-b} \| v \|_{p_b}$ using
  Lemma~\ref{thm:time-evolution-b}. Hence the second term can be
  controlled with a constant $C$ as
  \begin{equation*}
    \begin{split}
      &R_d(t)\, (1+t)^{-b_d} \beta_d(\sol Bst v)\,
      \| r_{\Theta} \|_{p_{1+t-s,b_0},0} \| \sol Bst v
      \|_{p_{1+t-s,b_0},0} \\
      &\le R_d(t)\, \| \sol Bst v \|^2_{p_{1+t-s,b_0},0} \\
      &\quad + C R_d(t)\, (1+t)^{-2b_d} (1+t-s)^{3-2b}
      \| v \|^2_{p_b} (1+t-s)^{2b_0} \| r_{\Theta} \|^2_{p_{b_0},0}.
    \end{split}
  \end{equation*}
  In both cases, the second term is uniformly integrable. This follows
  form the exponents as, in the case $b_d > 5/2$, we find
  \begin{equation*}
    -2b_d + 3 - 2b + 2b_0 \le -2_b + 2 = -3
  \end{equation*}
  and in the other case
  \begin{equation*}
    -2b_d + 3 - 2b + 2b_0
    < -\frac 12 - \bar{b}_d - b_d \le -2.
  \end{equation*}
  Therefore, there exists a constant constant $C$ such that
  \begin{equation*}
    \| \sol Bst v \|^2_{p_{1+t-s,b_0},0}
    \le C \| v \|^2_{p_{b},0}
    + C \int_s^t \| \sol Bs\tau v \|^2_{p_{1+t-s,b_0},0}\, \dd \tau.
  \end{equation*}
  The first control then shows the claimed result.
\end{proof}

For the difference between $\sol Bst$ and $\sol{L_1}st$, we adapt
Lemma~\ref{thm:alpha-control-integral-imaginary}.

\begin{lemma}
  \label{thm:k-q-stable-subspace}
  Let $b > 3/2$ and $\| r_{\Theta} \|_{p_b} < \infty$ and
  $\| v \|_{p_b} < \infty$. Then for $t \ge s$ it holds that
  \begin{equation*}
    \alpha(\sol Bst v - \sol{L_1}st v)
    +
    \int_0^t
    \Big(c_{\Theta,r} \Re (\sol Bst v {-} \sol{L_1}st v)_1(0)
    + c_{\Theta,i} \Im (\sol Bst v {-} \sol{L_1}st v)_1(0)\Big)
    \, \dd s
    = 0
  \end{equation*}
\end{lemma}
\begin{proof}
  As in Lemma \ref{thm:alpha-control-integral-imaginary}, note that
  \begin{equation*}
    \frac{\dd}{\dd t} \alpha(\sol Bst v - \sol{L_1}st v)
    = \alpha\left[
      L_1(\sol Bst v - \sol{L_1}st v)
      \right]
  \end{equation*}
  from where the result follows again.
\end{proof}

Now we prove the final control for the difference of the Volterra
kernel for the linearised evolution and the full evolution. Using the
previous Lemma~\ref{thm:higher-regularity-control}, moreover, the
difference is not adding a non-decaying contribution in the rotation
eigenmode, which is quantified by $K_{\Theta} \conv k_Q$.

\begin{lemma}
  \label{thm:control-k-q}
  Let $b_{\eta} \ge 0$, $b_d \ge 0$, and $b_r > b_{\eta}+2$ with
  \begin{equation*}
    \| r_{r} \|_{p_{b_r}} < \infty,\qquad
    \| r_{i} \|_{p_{b_r}} < \infty,\qquad
    \| r_{\Theta} \|_{p_{b_r}} < \infty,
  \end{equation*}
  and
  \begin{equation*}
    \| r_{r} \|_{p_{b_r-\frac 12},0} < \infty,\qquad
    \| r_{i} \|_{p_{b_r-\frac 12},0} < \infty,\qquad
    \| r_{\Theta} \|_{p_{b_r-\frac 12},0} < \infty.
  \end{equation*}
  Additionally assume one of the following conditions
  \begin{itemize}
  \item $b_d > 5/2$,
  \item $b_r > b_{\eta} + 4 + \max\{0,1-b_d\}$ and
    ${b_r > b_{\eta} + \frac{17}{4} + \max\{0,1-b_d\} +
    \frac{\max\left\{0,\frac 32 - b_d\right\}}{2} - \frac{b_d}{2}}$.
  \end{itemize}
  Then there exist constant $\delta_R$ and $C_Q$ such that for
  $J=[0,t]$ the deviation $k_Q = k - k_L$ is controlled as
  \begin{equation*}
    \tnorm{k_Q}_{L^\infty(J,p_{b_\eta})} \le C_Q \sqrt{R_d(t)}
  \end{equation*}
  and
  \begin{equation*}
    \tnorm{K_{\Theta} \conv k_Q}_{L^\infty(J,p_{b_\eta})} \le C_Q \sqrt{R_d(t)}
  \end{equation*}
  if $R_d(t) \le \delta_R$.
\end{lemma}
\begin{remark}
  The square root dependency of the bound on $R_d(t)$ can be
  improved. However, for the purpose of controlling the deviation it
  is sufficient.
\end{remark}
\begin{proof}
  Fix the initial time $s$, denote the two evolutions of $r_r$ as
  \begin{align*}
    v_r(t) &= \sol Bst r_r, \\
    w_r(t) &= \sol{L_1}st r_r
  \end{align*}
  and accordingly $v_i$ and $w_i$ with initial data $r_i$ such that
  \begin{equation*}
    k_Q(t,s) = -
    \begin{pmatrix} \setlength\arraycolsep{4pt}
      \Re (v_r-w_r)_1(t,0)
      & \Re (v_i-w_i)_1(t,0) \\
      \Im (v_r-w_r)_1(t,0)
      & \Im (v_i-w_i)_1(t,0) \\
    \end{pmatrix}.
  \end{equation*}
  In the following the proof works the same for the pairs $(v_r,w_r)$
  and $(v_i,w_i)$. Letting $(v,w)$ either pair we finish the proof.

  By the assumptions, we can choose $b$ and $b_0$ such that
  $b_0 > b_\eta + 3/2 + \max\{0,1-b_d\}$ and $b \ge b_0 + 1/2$ and
  $b_r \ge b$ satisfying the conditions of
  Lemma~\ref{thm:higher-regularity-control}. Moreover, we can ensure
  that $b \ge b_0 + 3/2$ if $b_d \le 5/2$. Hence by
  Lemma~\ref{thm:time-evolution-b} and
  \ref{thm:higher-regularity-control} there exist constants $\delta_R$
  and $C$ such that
  \begin{align*}
    \| v(t) \|_{p_{1+t-s,b_0},0} &\le C, \\
    \| w(t) \|_{p_{1+t-s,b_0},0} &\le C, \\
    \beta_d(v(t)) &\le (1+t-s)^{\frac 32 -b} C
  \end{align*}
  if $R_d(t) \le \delta_R$.

  From the definition
  \begin{equation*}
    \partial_t (v(t)-w(t))
    = L_1(v(t)-w(t)) + B_{1n}(v) + B_2(v).
  \end{equation*}
  Therefore, we find
  \begin{align*}
    \frac{\dd}{\dd t} \| v(t)-w(t) \|^2_{p_{1+t-s,b_0}}
    &\le 2 R_d(t)\, (1+t)^{-b_d}
    \| v(t) \|_{p_{1+t-s,b_0},0}
    \| v(t)-w(t) \|_{p_{1+t-s,b_0},0}\\
    &\quad+ 2 \beta_d(v(t))\, R_d(t)\, (1+t)^{-b_d}
    \| r_{\Theta} \|_{p_{1+t-s,b_0},0}
    \| v(t)-w(t) \|_{p_{1+t-s,b_0},0}.
  \end{align*}
  Using the control of $v$ and $w$, we find a constant $C$ such that
  \begin{equation*}
    \frac{\dd}{\dd t} \| v(t)-w(t) \|^2_{p_{1+t-s,b_0}}
    \le C R_d(t)\, (1+t)^{-b_d}
    \left[
      1 + (1+t-s)^{\frac 32 -b} (1+t-s)^{b_0}
    \right]
  \end{equation*}

  If $b_d > 5/2$, then the RHS is integrable and we find
  \begin{equation*}
    \| v(t) - w(t) \|_{p_{1+t-s,b_0}} \le C \sqrt{R_d(t)}
  \end{equation*}
  for a constant $C$ if $R_d(t) \le \delta_R(t)$.

  If $b_d \le 5/2$, then by the choice of $b$ we have
  $(1+t-s)^{\frac 32 -b} (1+t-s)^{b_0} \le 1$, so that
  \begin{equation*}
    \| v(t) - w(t) \|_{p_{1+t-s,b_0}} \le C\, (1+t-s)^{\max\{1-b_d,0\}} \sqrt{R_d(t)}
  \end{equation*}
  for a constant $C$ if $R_d(t) \le \delta_R(t)$.

  By Lemma~\ref{thm:control-seminorms-by-weighted-sobolev}, this gives us
  directly the pointwise control
  \begin{equation*}
    \beta_{\eta}(v(t)-w(t))
    \le C\, (1+t-s)^{-b_0+\max\{1-b_d,0\}} \sqrt{R_d(t)}
  \end{equation*}
  and
  \begin{equation}
    \label{eq:k-q-bounds-beta-alpha}
    \beta_{\alpha}\Big(v(t)-w(t)\Big)
    \le C\, (1+t-s)^{\frac 12-b_0+\max\{1-b_d,0\}}  \sqrt{R_d(t)}.
  \end{equation}

  The first part of the lemma follows by integration for $t \ge s$ as
  \begin{equation*}
    (1+t)^{b_\eta} \int_0^t \|k_Q(t,s)\|\, (1+s)^{-b_\eta} \, \dd s
    \le C \sqrt{R_d(t)} \int_0^t (1+t-s)^{b_\eta-b_0+\max\{1-b_d,0\}} \, \dd s,
  \end{equation*}
  which is uniformly bounded.

  For the effect on the rotation eigenmode, we find explicitly
  \begin{equation*}
    (K_\Theta  \conv k_Q)(t,s)
    = -
    \begin{pmatrix}
      0 \\ 1
    \end{pmatrix}
    \int_s^t
    \begin{pmatrix}
      c_{\Theta,r} & c_{\Theta,i}
    \end{pmatrix}
    \begin{pmatrix}
      \Re (\sol Bs\tau r_r {-} \sol {L_1}s\tau r_r)_1(0) &
      \Re (\sol Bs\tau r_i {-} \sol {L_1}s\tau r_i)_1(0) \\
      \Im (\sol Bs\tau r_r {-} \sol {L_1}s\tau r_r)_1(0) &
      \Im (\sol Bs\tau r_i {-} \sol {L_1}s\tau r_i)_1(0)
    \end{pmatrix}
    \dd \tau.
  \end{equation*}
  The integral can be identified with Lemma \ref{thm:k-q-stable-subspace}
  and controlled by \eqref{eq:k-q-bounds-beta-alpha} as
  \begin{equation*}
    \| (K_{\Theta} \conv k_{Q})(t,s) \|
    \le C (1+t-s)^{\frac 12-b_0+\max\{1-b_d,0\}}  \sqrt{R_d(t)}.
  \end{equation*}
  By the choice of $b_0$, this bound gives the claimed control on
  $K_{\Theta} \conv k_Q$.
\end{proof}

\section{Estimate for the order parameter}
\label{sec:estimate-order-parameter}

For the product $\conv$ and the norm
\eqref{eq:def-general-kernel-norm}, we collect the basic properties.
\begin{lemma}
  \label{thm:volterra-kernel-algebra-bounds}
  Let $\beta,\gamma \in \vkV(J,\phi)$. Then
  \begin{equation*}
    \tnorm{\beta \conv \gamma}_{L^\infty(J,\phi)}
    \le \tnorm{\beta}_{L^\infty(J,\phi)} \tnorm{\gamma}_{L^\infty(J,\phi)}.
  \end{equation*}

  If $\beta(t,s) = \beta_{c}(t-s)$ then
  \begin{equation*}
    \tnorm{\beta}_{L^\infty(J,\phi)}
    \le \| \beta_c \|_{L^1(J,\phi)}
  \end{equation*}
  and for a function $F$ on $J$
  \begin{equation*}
    \| (\beta \conv F)(t) \|_{L^\infty(J,\phi)}
    \le \tnorm{\beta}_{L^\infty(J,\phi)}
    \| F \|_{L^\infty(J,\phi)}.
  \end{equation*}
\end{lemma}
\begin{proof}
  The inequalities follow directly from the submultiplicativity, see
  \cite[Section~2 of
  Chapter~9]{gripenberg-londen-staffans-1990-volterra}.
\end{proof}

The general Volterra equation \eqref{eq:volterrra-eta} with kernel \(k\)
can be solved using a general resolvent \(r\) satisfying
\begin{equation}
  \label{eq:nonconvolution-resolvent-definition}
  r + k \conv r = r + r \conv k = k.
\end{equation}

In the linear case, we have identified the resolvent
\(r_L(t,s) = r_{Lc}(t-s)\) using the convolution structure
(Lemma~\ref{thm:resolvent-eigenmode-form}). For a small enough
perturbation such a resolvent can be constructed by a series.  Here we
face the difficulty that the resolvent $r_{Lc}$ is not in
$L^1(\R,p_{b_{\eta}})$ due to the rotation eigenmode contribution
$K_{\Theta}$. However, we can circumvent the problem by only using
$r_{L} \conv k_{Q}$, which is better behaved, because the kernel $k_Q$
only creates contributions in the stable subspace. Adapting Lemma~3.7
of Chapter~9 of \cite{gripenberg-londen-staffans-1990-volterra} to the
case of an eigenmode, we find the resolvent.
\begin{lemma}
  \label{thm:perturbation-resolvent}
  Let $b_{\eta} \ge 0$ and let $k_{Lc}$ be a convolution kernel with
  $k_{Lc} \in L^1(\R^+,p_{b_\eta})$ and resolvent
  \begin{equation*}
    r_{Lc} = K_{\Theta} + r_{Lcs}
  \end{equation*}
  where $K_{\Theta}$ is a constant matrix and
  $r_{Lcs} \in L^1(\R^+,p_{b_{\eta}})$. Let $r_L$ and $r_{Ls}$ be the
  Volterra kernels corresponding to the convolution kernels $r_{Lc}$
  and $r_{Lcs}$, respectively.

  Assume a Volterra kernel $k_Q$ for a time range $J = [0,T]$
  satisfying
  \begin{equation*}
    \tnorm{k_Q - r_{L} \conv k_Q}_{L^\infty(J,p_{b_{\eta}})} < 1.
  \end{equation*}
  Then $k = k_L + k_Q$ has a resolvent $r$ for the time range $J$ of
  the form
  \begin{equation*}
    r = K_{\Theta} + r_Q \conv K_{\Theta} + r_{s},
  \end{equation*}
  where
  \begin{equation*}
    \tnorm{r_Q}_{L^\infty(J,p_{b_\eta})} \le
    \frac{\tnorm{k_Q - r_{L} \conv k_Q}_{L^\infty(J,p_{b_{\eta}})}}
    {1-\tnorm{k_Q - r_{L} \conv k_Q}_{L^\infty(J,p_{b_{\eta}})}}
  \end{equation*}
  and
  \begin{equation*}
    \tnorm{r_s}_{L^\infty(J,p_{b_\eta})} \le
    \frac{\tnorm{k_Q - r_{L} \conv k_Q +
      r_{Ls}}_{L^\infty(J,p_{b_{\eta}})}}
    {1-\tnorm{k_Q - r_{L} \conv k_Q}_{L^\infty(J,p_{b_{\eta}})}}.
  \end{equation*}
\end{lemma}
\begin{proof}
  We define $r$ by
  \begin{equation*}
    r = (k - r_{L} \conv k) +
    \left(\sum_{n=1}^{\infty} (-1)^n (k_Q - r_{L} \conv k_Q)^{\conv n} \right)
    \conv (k - r_{L} \conv k),
  \end{equation*}
  which is an absolutely converging sum by the assumed
  bound. Moreover, we note that
  \begin{equation*}
    k - r_{L} \conv k
    = K_{\Theta} + (k_Q - r_{L} \conv k_Q + r_{Ls}),
  \end{equation*}
  so that $r$ has the claimed form with the bounds of $r_Q$ and $r_s$.

  In order to show that $r$ satisfies
  \eqref{eq:nonconvolution-resolvent-definition}, we first note that
  multiplying by $k_Q-r_L\conv k_Q$ from the left shows that
  \begin{equation}
    \label{eq:perturbation-kernel-right-1}
    (k_Q-r_L \conv k_Q) \conv r = -r + k - r_L \conv k.
  \end{equation}
  Multiplying \eqref{eq:perturbation-kernel-right-1} by $k_L$ from the
  left shows together for the resolvent equation for the kernel $k_L$
  that
  \begin{equation}
    \label{eq:perturbation-kernel-right-2}
    (k_L + r_L \conv k_Q) \conv r = r_L \conv k.
  \end{equation}
  Combining \eqref{eq:perturbation-kernel-right-1} and
  \eqref{eq:perturbation-kernel-right-2} then gives
  \begin{equation*}
    r + k \conv r = k.
  \end{equation*}

  For the other part of the resolvent equation
  \eqref{eq:nonconvolution-resolvent-definition}, note that from the
  definition of $r$ it follows that
  \begin{equation}
    \label{eq:perturbation-kernel-left}
    r + r \conv k_L = k_L - \left(\sum_{n=1}^{\infty} (-1)^n (k_Q - r_{L} \conv k_Q)^{\conv n} \right).
  \end{equation}
  Multiplying \eqref{eq:perturbation-kernel-left} by $k_Q-r_L \conv k_Q$ then shows that
  \begin{equation*}
    \begin{split}
      &(r + r \conv k_L) \conv (k_Q-r_L \conv k_Q) \\
      &= k_L \conv (k_Q-r_L \conv k_Q)
      + \left(\sum_{n=1}^{\infty} (-1)^n (k_Q - r_{L} \conv k_Q)^{\conv
          n} \right)
      + (k_Q - r_L \conv k_Q).
    \end{split}
  \end{equation*}
  Replacing the sum by \eqref{eq:perturbation-kernel-left} and
  rearranging then shows the required equality.
  \begin{equation*}
    r + r \conv k = k. \qedhere
  \end{equation*}
\end{proof}

Combining this result with the bound on $k_Q$, we prove the result for
the Volterra equation.
\begin{proof}[Proof of Lemma~\ref{thm:control-eta-volterra}]
  By Proposition~\ref{thm:linear-resolvent-bound}, the linear convolution kernel
  $k_{Lc}$ has a resolvent
  \begin{equation*}
    r_{Lc} = K_{\Theta} + r_{Lcs},
  \end{equation*}
  with $\| r_{Lcs} \|_{L^1(\R^+,p{_{b_{\eta}}})} < \infty$.

  For the nonlinear behaviour, we fix the time range $J = [0,t]$ and
  assume that $R_d(t) \le \delta_R$, where $\delta_R$ is chosen small
  enough so that Lemma \ref{thm:control-k-q} implies
  \begin{align*}
    \tnorm{ k_Q }_{L^\infty(J,p_{b_{\eta}})} &\le \frac{2 + \| r_{Lcs}
    \|_{L^1(\R^+,p{_{b_{\eta}}})}}{4}, \\
    \tnorm{ K_{\Theta} \conv k_Q }_{L^\infty(J,p_{b_{\eta}})} &\le \frac{2 + \| r_{Lcs}
    \|_{L^1(\R^+,p{_{b_{\eta}}})}}{4}.
  \end{align*}

  Then Lemma~\ref{thm:volterra-kernel-algebra-bounds} implies that
  \begin{equation*}
    \tnorm{k_Q - r_L \conv k_Q}_{L^\infty(J,p_{b_{\eta}})}
    \le (1 + \tnorm{r_{Ls}}_{L^\infty(J,p_{b_{\eta}})})\, \tnorm{k_Q}_{L^\infty(J,p_{b_{\eta}})}
    + \tnorm{K_{\Theta} \conv k_Q}_{L^\infty(J,p_{b_{\eta}})}
    \le \frac 12.
  \end{equation*}

  Therefore, Lemma \ref{thm:perturbation-resolvent} shows that the kernel
  $k$ has the resolvent
  \begin{equation*}
    r = K_{\Theta} + r_Q \conv K_{\Theta} + r_{s},
  \end{equation*}
  where
  \begin{equation*}
    \tnorm{r_Q}_{L^\infty(J,p_{b_\eta})} \le 1
    \qquad\text{ and }\qquad
    \tnorm{r_s}_{L^\infty(J,p_{b_\eta})} \le 1 + 2\, \| r_{Lcs} \|_{L^1(\R^+,p{_{b_{\eta}}})}.
  \end{equation*}
  By Theorem~3.6 of Chapter~9 of
  \cite{gripenberg-londen-staffans-1990-volterra}, the Volterra
  equation \eqref{eq:volterrra-eta} then has over the time range $J$
  the unique solution
  \begin{equation}
    \label{eq:solution-volterra-eta}
    \begin{pmatrix}
      \Re \conj{\eta} \\ \Im \conj{\eta}
    \end{pmatrix}
    =
    \compF - r \conv \compF.
  \end{equation}
  This follows by elementary calculations from
  \eqref{eq:nonconvolution-resolvent-definition}, which we repeat
  here. Indeed \eqref{eq:solution-volterra-eta} defines a solution,
  because
  \begin{equation*}
    \compF - r \conv \compF
    + k \conv (\compF - r \conv \compF)
    = \compF + (k - r - k\conv r) \conv \compF
    = \compF.
  \end{equation*}
  On the other hand for a solution
  \begin{equation*}
    \begin{pmatrix}
      \Re \conj{\eta} \\ \Im \conj{\eta}
    \end{pmatrix}
    + k \conv
    \begin{pmatrix}
      \Re \conj{\eta} \\ \Im \conj{\eta}
    \end{pmatrix}
    = F,
  \end{equation*}
  we find by multiplying from the left with $r$ that
  \begin{equation*}
    k \conv
        \begin{pmatrix}
      \Re \conj{\eta} \\ \Im \conj{\eta}
    \end{pmatrix}
    = r \conv F,
  \end{equation*}
  which shows that the solution has the claimed form
  \eqref{eq:solution-volterra-eta}.

  By Corollary \ref{thm:forcing-control-eigenmode-b-alpha}, we find
  \begin{equation*}
    \| K_{\Theta} \conv F \|_{L^\infty(J,p_{\eta})}
    \le \sup_{s\in J} \beta_{\alpha}(\sol B0t u_{\init}) \,(1+s)^{b_{\eta}}.
  \end{equation*}
  Hence by Lemma \ref{thm:volterra-kernel-algebra-bounds},
  \begin{equation*}
    \| r \conv F \|_{L^\infty(J,p_{\eta})}
    \le \Big(1+\tnorm{r_Q}_{L^\infty(J,p_{b_\eta})}\Big)
    \| K_{\Theta} \conv F \|_{L^\infty(J,p_{\eta})}
    + \tnorm{r_s}_{L^\infty(J,p_{b_\eta})}
    \| F \|_{L^\infty(J,p_{\eta})},
  \end{equation*}
  which is the claimed control.
\end{proof}

\section{Bootstrap argument}
\label{sec:bootstrap}

The obtained estimates allow a control of the order parameter
$\eta(t)$.
\begin{proof}[Proof of Lemma \ref{thm:bootstrap-order-parameter}]
  First we chose $\delta_R$ small enough to apply
  Lemma~\ref{thm:time-evolution-b} to conclude under $R_d(t) \le \delta_R$ that
  \begin{equation*}
    \beta_d(\sol B0t u_\init)\, (1+t)^{-\frac 32 + b}
    \le 2 C_{\beta d} \| u_{\init} \|_{p_b}.
  \end{equation*}

  With this control we find by Lemma~\ref{thm:time-evolution-b}
  \begin{align*}
    \beta_\alpha(\sol B0t u_\init)\, (1{+}t)^{-\frac 12 +b}
    &\le (2b{-}1)^{-1/2} \| u_\init \|_{p_b} \\
    &+ \frac{2 C_{\beta d}}{\sqrt{2b_r{-}1}}
    \| r_{\Theta} \|_{p_{b_r}} \|K_{\Theta}\| R_d(t)
    (1{+}t)^{-\frac 12 +b}
    \int_0^t (1{+}t{-}s)^{\frac 12 - b_r} (1{+}s)^{-b_d + \frac 32 -b}\, \dd s.
  \end{align*}
  Here we find
  \begin{equation*}
    (1{+}t)^{-\frac 12 +b}
    \int_0^t (1{+}t{-}s)^{\frac 12 - b_r} (1{+}s)^{-b_d + \frac 32 -b}
    \,\dd s
    \le
    \int_0^t (1{+}t{-}s)^{b - b_r} (1{+}s)^{1-b_d}\, \dd s,
  \end{equation*}
  which is bounded uniformly over $t$, as $b_d > 1$ and $b_r > b+1$.

  Likewise, it holds that
  \begin{align*}
    &\beta_\eta(\sol B0t u_\init)\, (1{+}t)^{-\frac 12 +b} \\
    &\le C_S \| u_\init \|_{p_b}
    + 2 C_{\beta d} C_S
    \| r_{\Theta} \|_{p_{b_r}} R_d(t)
    (1{+}t)^{-\frac 12 +b}
    \int_0^t (1{+}t{-}s)^{- b_r} (1{+}s)^{-b_d + \frac 32 -b}\, \dd s,
  \end{align*}
  which can be controlled in the same way. Hence there exists a
  constant $C$ such that $F(t) = (\sol B0t u_\init)_1(0)$ satisfies
  with \(b_\eta = b - 1/2\)
  \begin{equation*}
    \sup_{s \in [0,t]}
    \Big(
      |F(s)| + \beta_{\alpha}\left(\sol B0s u_{\init}\right)
    \Big)
    (1+s)^{b_{\eta}}
    \le C \| u_{\init} \|_{p_{b}}
  \end{equation*}
  if $R_d(t) \le \delta_R$.

  Decreasing $\delta_R$ if needed, we apply
  Lemma \ref{thm:control-eta-volterra} for the Volterra equation with
  $b_{\eta} = b - 1/2$ and obtain the result.
\end{proof}

With the order parameter, we can control the full solution.
\begin{proof}[Proof of Lemma \ref{thm:bootstrap-beta-d}]
  By Lemma~\ref{thm:time-evolution-b}, the solution up to time \(t\)
  is given by \eqref{eq:solution-u-duhamel}. Taking \(\beta_d\) of the
  equation yields
  \begin{equation*}
    \beta_d(u(t)) \le \beta_d(\sol B0t u_{\init})
    + \int_0^t \beta_d(\sol Bst L_2 u(s))\, \dd s.
  \end{equation*}
  Recall that
  \begin{equation*}
    L_2 u(s) = \Re \conj{\eta}(s)\, r_r
    + \Im \conj{\eta}(s)\, r_i
  \end{equation*}
  By choosing $\delta_R$ small enough, Lemma \ref{thm:time-evolution-b}
  implies under $R_d(t) \le \delta_R$ that
  \begin{equation*}
    \beta_d(u(t)) \le 2C_{\beta d} (1{+}t)^{\frac 32-b}
    \| u_{\init} \|_{p_{b}}
    + 2 C_{\beta d} C_{r}
    \int_0^t (1{+}t{-}s)^{\frac 32-b_r} (1{+}s)^{-b_d} \dd s
    \sup_{s \in [0,t]} (1{+}s)^{b_d} |\eta(s)|,
  \end{equation*}
  where
  \begin{equation*}
    C_{r} = \| r_{r} \|_{p_{b_r}} + \| r_{i} \|_{p_{b_r}}.
  \end{equation*}
  As $b-b_r < -1$ and $-3/2 + b - b_d = -1$, there exists a constant
  $C$ such that for all time $t$ it holds that
  \begin{equation*}
    (1+t)^{-\frac 32 + b}
    \int_0^t (1+t-s)^{\frac 32-b_r} (1+s)^{-b_d} \dd s
    \le
    \int_0^t (1+t-s)^{b-b_r} (1+s)^{-\frac 32 + b - b_d}
    \le C,
  \end{equation*}
  which shows the claimed control.
\end{proof}

In preparation of the bootstrap argument, we prove a well-posedness
result.
\begin{lemma}
  \label{thm:local-well-posedness}
  Let $b \ge 0$ and $\hat{f}_{\init}$ be initial data. Assume that the
  velocity marginal $\hat{g} = (\hat{f}_{\init})_0$ satisfies
  $\|\hat{g}\|_{p_b} < \infty$ and that the restriction to
  $\ell \ge 1$ is $\hat{f}_{\init} \in \sX_{p_b}$. Then for any time
  $T > 0$, there exists a global unique solution
  $\hat{f} \in C_w([0,T],\sX_{p_{b}})$ to
  \eqref{eq:kuramoto-mean-field-pde-fourier} with initial data
  $\hat{f}_{\init}$ and the constant velocity marginal $\hat{g}$.
\end{lemma}
\begin{proof}
  Given a solution $\hat{f} \in C_w([0,T],\sX_{p_{b}})$, we find that
  $\hat{f}$ is a continuous function of time $t$ and frequency $\xi$
  by Morray's inequality. Hence Theorem 15 in
  \cite{dietert-2016-stability-bifurcation} applies to show
  uniqueness.

  The existence can be proven similar to Lemma~3.2 of
  \cite{dietert-fernandez-gerard-varet-2018-landau-damping-pls}. The
  key point is the a priori estimate
  \begin{equation*}
    \partial_t \| \hat{f} \|^2_{p_b}
    \le K |\hat{f}_1(0)| \| g \|_{p_b} \| \hat{f} \|_{p_b},
  \end{equation*}
  similar to the estimate in the proof of
  Lemma~\ref{thm:evolution-l1-b1}. By the Sobolev embedding used in
  Lemma~\ref{thm:control-seminorms-by-weighted-sobolev}, this shows that
  there exists a constant $C$ such that
  \begin{equation*}
    \partial_t \| \hat{f} \|^2_{p_b} \le C \| \hat{f} \|^2_{p_b}.
  \end{equation*}
  As in
  \cite{dietert-fernandez-gerard-varet-2018-landau-damping-pls},
  we can build approximate solutions $\hat{f}^n$ satisfying this bound
  by restricting to the spatial modes $\ell \in [1,\dots ,n]$ and
  compact smooth initial data.

  Looking at $\ell = 1$ and $\xi \in [0,1]$, the estimate shows that
  $\hat{f}^n_1 \in L^\infty([0,T],H^1([0,1]))$ and
  $\partial_t \hat{f}^n_1 \in L^\infty([0,T],L^2([0,1]))$. By the
  Aubin-Lions Lemma, we extract a subsequence for which
  $\hat{f}^n_1(\cdot,0)$ converges strongly in $L^\infty([0,T])$. By
  the weak compactness, we can extract a further subsequence
  converging to a weak solution.
\end{proof}

For the initial data $\hat{f}_{\init}$, we show that close to
$\hat{f}_{\stat}$ it can be written in the required polar coordinate
form.
\begin{lemma}
  \label{thm:local-choice-theta}
  Given a stationary state $\hat{f}_{\stat}$ and $b > 3/2$. Then there
  exist $\delta_{\init}$ and $\delta_{\Theta}$ such that for
  $\hat{f}_{\init}$ with
  $\| \hat{f}_{\init} - \hat{f}_{\stat} \|_{p_b} < \delta_\init$,
  there exists a unique $\Theta \in (-\delta_\Theta,\delta_{\Theta})$
  such that $u = R_{-\Theta} \hat{f}_{\init} - \hat{f}_{\stat}$
  satisfies
  \begin{equation*}
    \alpha(u) = 0.
  \end{equation*}
  Moreover, there exists a constant $C$ such that
  \begin{equation*}
    | \Theta | \le C \, \| \hat{f}_{\init} - \hat{f}_{\stat} \|_{p_b}
  \end{equation*}
  and
  \begin{equation*}
    \| u \|_{p_b} \le \| f_{\init} - f_{\stat} \|_{p_b} + |\Theta| \,
    \| \hat{f}_\stat \|_{p_b,0}.
  \end{equation*}
\end{lemma}
\begin{proof}
  Define the function $F : \T \times \sX_{p_b} \mapsto \R$ by
  \begin{equation*}
    F(\Theta,\hat{f}_{\init}) =
    \alpha(R_{-\Theta} \hat{f}_{\init} - \hat{f}_{\stat}).
  \end{equation*}
  By Lemma~\ref{thm:control-seminorms-by-weighted-sobolev} the function $F$
  is continuously differentiable and
  \begin{equation*}
    \partial_\Theta F(0,\hat{f}_{\stat}) = - \alpha(\DD \hat{R}
    \hat{f}_{\stat}) \not = 0.
  \end{equation*}
  Hence by the implicit function theorem, there exists a unique
  inverse in the neighbourhood of $\hat{f}_{\stat}$ with the given
  control on \(\Theta\).

  For the control on $u$ note that
  \begin{equation*}
    \| u \|_{p_b} = \| \hat{R}_{\Theta} u \|_{p_b}
    \le \| \hat{f}_{\init} - \hat{f}_{\stat} \|_{p_b}
    + \| \hat{R}_{\Theta} \hat{f}_{\stat} - \hat{f}_{\stat} \|_{p_b}.
  \end{equation*}
  As $|\ee^{i\ell \Theta} - 1| \le \ell |\Theta|$, the second term can be
  bounded as
  \begin{equation*}
    \| \hat{R}_{\Theta} \hat{f}_{\stat} - \hat{f}_{\stat} \|_{p_b}
    \le |\Theta|\, \| \hat{f}_{\stat} \|_{p_b,0},
  \end{equation*}
  which is the claimed result.
\end{proof}

With this we can assemble the proof of the main theorem.
\begin{proof}[Proof of Theorem \ref{thm:nonlinear-control-final}]
  By Proposition~\ref{thm:regularity-stationary-state}, we find
  \begin{align*}
    \| \hat{f}_\stat \|_{p_{b_g-2},1}
    \le C\, \| \hat{g} \|_{p_{b_g}}, \quad\text{and}\quad
    \| \hat{f}_\stat \|_{p_{b_g-3/2},1/2}
    \le C\, \| \hat{g} \|_{p_{b_g}}.
  \end{align*}
  for a constant \(C\).  The definitions of \(r_r\), \(r_i\) and
  \(r_{\Theta}\) then imply
  \begin{equation*}
    \| r_{r} \|_{p_{b_r}} < \infty,\qquad
    \| r_{i} \|_{p_{b_r}} < \infty,\qquad
    \| r_{\Theta} \|_{p_{b_r}} < \infty,
  \end{equation*}
  and
  \begin{equation*}
    \| r_{r} \|_{p_{b_r-\frac 12},0} < \infty,\qquad
    \| r_{i} \|_{p_{b_r-\frac 12},0} < \infty,\qquad
    \| r_{\Theta} \|_{p_{b_r-\frac 12},0} < \infty
  \end{equation*}
  with \(b_r = b_g - 3/2\).

  By Lemma \ref{thm:local-well-posedness}, there exists a global weak
  solution $\hat{f}$, which is locally bounded in $\sX_{p_b}$ and
  weakly continuous.

  By Lemma \ref{thm:local-choice-theta}, we can choose $\delta$ small
  enough so that there exists an initial angle $\Theta(0)$ such that
  \begin{equation*}
    |\Theta(0)| \le C\, \| f_{\init} - \hat{f}_\stat \|_{p_b}
  \end{equation*}
  and for a constant $C_f$
  \begin{equation*}
    \| u_{\init} \|_{p_b} \le C_f\, \| f_{\init} - f_{\stat} \|_{p_b},
  \end{equation*}
  as the bound of $\| r_{\Theta} \|_{p_b}$ implies the control of
  $\|\hat{f}_{\stat}\|_{p_b,0}$.

  By setting $u = R_{-\Theta} \hat{f} - f_{\stat}$ and evolving
  $\Theta$ by
  \begin{equation*}
    \frac{\dd}{\dd t} \Theta = \dot{\Theta}
  \end{equation*}
  given by \eqref{eq:def-theta-dot}, we find an evolution for $u$ and
  $\Theta$ as long as
  \begin{equation}
    \label{eq:bootstrap-condition-polar-coordinates}
    \beta_d(u(t)) \le \frac 12 |\alpha(\DD \hat{R} \hat{f}_{\stat})|,
  \end{equation}
  because $|\alpha(\DD \hat{R} u)| \le \beta_d(u(t))$. Under this
  assumption, $u$ is locally bounded and weakly continuous, and
  $\dot{\Theta}$ is continuous. Moreover, if
  \eqref{eq:bootstrap-condition-polar-coordinates} holds up to a time
  $t$, the solution can be extended by a positive amount of time.

  Then $u$ is given by \eqref{eq:solution-u-duhamel}, because the
  evolution under $B$ has a unique solution and $L_2$ is a bounded
  operator. With the control $R_d$ on the coefficients, the estimates
  from Lemma \ref{thm:bootstrap-order-parameter} and Lemma \ref{thm:bootstrap-beta-d} show
  that there exist constants $\delta_R$ and $C$ such that
  \begin{align*}
    |\eta(t)| &\le C\, (1+t)^{-b_d}\, \|u_\init\|_{p_b},\\
    |\beta_d(u(t))| &\le C\, (1+t)^{\frac 32-b}\, \|u_\init\|_{p_b}
  \end{align*}
  if $R_d(t) \le \delta_R$. This in particular implies the existence of
  a constant $C_Q$ with
  \begin{equation*}
    |\alpha(Q(u(t)))| \le K\, \| K_\Theta\|\, |\eta(t)|\, \beta_d(u(t))
    \le C_Q (1+t)^{\frac 32 - b - b_d} \| u_{\init} \|^2_{p_{b}}
  \end{equation*}
  and the existence of $\delta_{\Theta}$ such that
  \begin{equation*}
    \beta_d(u(t)) \le \frac 14 |\alpha(\DD \hat{R} \hat{f}_{\stat})|
  \end{equation*}
  if $R(t) \le \delta_R$ and $\| u_{\init} \| \le \delta_\Theta$.

  Therefore, we can find a small enough $\delta$ such that for $\|
  u_{\init} \| \le \delta$, the bootstrap assumption $R_d(t) \le
  \delta_R$ implies
  \begin{equation*}
    \beta_d(u(t)) \le \frac 14 |\alpha(\DD \hat{R} \hat{f}_{\stat})|
  \end{equation*}
  and
  \begin{equation*}
    R_d(t) \le \frac 12 \delta_R.
  \end{equation*}
  By the continuity of the solution and the local well-posedness, this
  implies the existence of a global solution $u$ and $\Theta$
  satisfying $R_d(t) \le \delta_R$ for all times $t$. Then, in
  particular, the stated bounds on $\eta(t)$ and $\dot{\Theta}$ hold.
\end{proof}

\section*{Acknowledgements}

The author would like to thank David Gérard-Varet and Bastien
Fernandez for the very helpful discussions during the work.

\section*{Ethical statement}

Funding: The author was supported for the study by the ANR Chaire
d'Excellence ANR-11-IDEX-005 and the People Programme (Marie Curie
Actions) of the European Union’s Seventh Framework Programme
(FP7/2007-2013) under REA grant agreement n.\ PCOFUND-GA-2013-609102,
through the PRESTIGE programme coordinated by Campus France.

Conflict of Interest: The author previously was previously supported
by the UK Engineering and Physical Sciences Research Council (EPSRC)
grant EP/H023348/1 for the University of Cambridge Centre for Doctoral
Training, the Cambridge Centre for Analysis.

\printbibliography
\end{document}